\documentclass[12pt]{amsart}

\usepackage[margin=1in, marginparsep=1pt,
            marginparwidth=124pt]{geometry}

\usepackage[english]{babel}
\usepackage[usenames,dvipsnames,svgnames,table]{xcolor}
\usepackage[pagebackref,linktocpage=true,colorlinks=true,linkcolor=Blue,citecolor=BrickRed,urlcolor=RoyalBlue]{hyperref}
\usepackage[alphabetic]{amsrefs}

\usepackage[colorinlistoftodos,prependcaption,textsize=tiny]{todonotes}
\renewcommand{\ge}{\geqslant}

\renewcommand{\le}{\leqslant}

\usepackage{amsmath,amssymb,amsfonts,amsthm,enumerate,textcomp}
\usepackage{url}
\usepackage{graphicx}
\usepackage{xcolor}
\usepackage{wrapfig}

\usepackage[mathscr]{euscript} 

\usepackage{esint}
\usepackage{enumitem}

\theoremstyle{plain}
\newtheorem{theorem}{Theorem}[section]

\newtheorem{lemma}[theorem]{Lemma}
\newtheorem{prop}[theorem]{Proposition}
\newtheorem{remark}[theorem]{Remark}
\newtheorem{corollary}[theorem]{Corollary}

\usepackage{amsmath}
\usepackage{amssymb}

\numberwithin{equation}{section}

\newcommand\R{\mathbb R} 
\newcommand\N{\mathbb N} 
\renewcommand\S{\mathbb S} 

\newcommand\e{\varepsilon}
\renewcommand{\epsilon}{\e}

\usepackage{tikz}
\usetikzlibrary{arrows,shapes,positioning}
\usetikzlibrary{decorations.markings}
\tikzstyle arrowstyle=[scale=1]
\tikzstyle directed=[postaction={decorate,decoration={markings,
    mark=at position .65 with {\arrow[arrowstyle]{stealth}}}}]
\tikzstyle reverse directed=[postaction={decorate,decoration={markings,
    mark=at position .65 with {\arrowreversed[arrowstyle]{stealth};}}}]

\allowdisplaybreaks

\thanks{
\\
It is a pleasure to thank Xavier Ros-Oton for several interesting comments
on a preliminary version of this paper.\\
SD was supported by the Australian Research Council DECRA DE180100957 {\em PDEs, free boundaries and applications}.
AK was supported by the EPSRC grant 
EP/S03157X/1 {\em Mean curvature measure of free boundary}.
EV was supported by the Australian Laureate Fellowship FL190100081 {\em Minimal surfaces, free boundaries and
partial differential equations}.}

\begin{document}

\title[Global solutions of a free boundary problem]{Classification of global solutions \\ of a free boundary problem in the plane}

\author{Serena Dipierro}
\address{Serena Dipierro: Department of Mathematics
and Statistics,
University of Western Australia,
35 Stirling Hwy, Crawley WA 6009, Australia}
\email{serena.dipierro@uwa.edu.au}

\author{Aram Karakhanyan}
\address{Aram Karakhanyan:
School of Mathematics, The University of Edinburgh,
Peter Tait Guthrie Road, EH9 3FD Edinburgh, UK}
\email{aram.karakhanyan@ed.ac.uk}

\author{Enrico Valdinoci}
\address{Enrico Valdinoci:
Department of Mathematics
and Statistics,
University of Western Australia,
35 Stirling Hwy, Crawley WA 6009, Australia}
\email{enrico.valdinoci@uwa.edu.au}

\begin{abstract}
We classify nontrivial, nonnegative, positively homogeneous solutions of the equation
\begin{equation*}
\Delta u=\gamma u^{\gamma-1}
\end{equation*}
in the plane.

The problem is motivated by the analysis of the classical 
Alt-Phillips free boundary problem, but considered here with negative exponents~$\gamma$.

The proof relies on several bespoke results for ordinary differential equations.
\end{abstract}

\date{\hbox{\today}}
\maketitle

\setcounter{tocdepth}{2}
\tableofcontents

\section{Introduction}

Several problems of interest in the calculus of variations can be reduced to the study of critical points
of an energy functional of the type
$$ \int \frac{|\nabla u|^2}{2} +F(u)\,$$
where, up to a normalization, $F(r)\ge0$ for all~$r\in\R$ and~$F(r)=0$ for all~$r\in(-\infty,0]$.

An archetypal example of the potential~$F$ is given by power-like functions such as
\begin{equation}\label{LEQF}
F(r):=r^\gamma\,\chi_{(0,+\infty)}(r),\end{equation}
for a given~$\gamma\in\R$. In this case, nonnegative critical points of the energy functional
formally correspond to solutions of the equation
\begin{equation}\label{LEQ} \Delta u=\gamma u^{\gamma-1}\end{equation}
in~$\{u>0\}$.

When~$\gamma\ge2$, we have that~$F\in C^{1,1}(\R)$ and the right hand side of~\eqref{LEQ} is Lipschitz continuous
in~$u$. In particular, in this case one can define~$c:=-\gamma u^{\gamma-2}$ and deduce that~$c$ is continuous
if so is~$u$: in this setting, the Strong Maximum Principle (see e.g. Theorem~1.7 in~\cite{MR2777537})
yields that nonnegative solutions of~\eqref{LEQ} are actually strictly positive inside the domain in which the equation
takes place.

When~$\gamma\in(0,2)$, the situation changes significantly: for instance, it is readily checked that
\begin{equation}\label{PKS93}u(x):=
\left(\frac{(2-\gamma)^2\,(x_n)_+^2}{2}
\right)^{\frac1{2-\gamma}}
\end{equation} is in this case Lipschitz continuous and, for every~$\phi\in C^\infty_0(B_1)$, the partial integration yields the identity
\begin{eqnarray*}&& \int_{B_1} \nabla u(x)\cdot\nabla\phi(x)+\gamma u^{\gamma-1}(x)\phi(x)\,dx
=\int_{B_1\cap\{x_n>0\}} \nabla u(x)\cdot\nabla\phi(x)+\gamma u^{\gamma-1}(x)\phi(x)\,dx\\&&\qquad=
\int_{B_1\cap\{x_n>0\}} 
\frac{(2-\gamma)^{\frac\gamma{2-\gamma}}}{2^{\frac{\gamma-1}{2-\gamma}}}
x_n^{\frac\gamma{2-\gamma}}
\,\partial_n\phi(x)+
\gamma \,\left(\frac{(2-\gamma)^2\,x_n^2}{2}
\right)^{\frac{\gamma-1}{2-\gamma}}
\,
\phi(x)\,dx\\&&\qquad=
\int_{B_1\cap\{x_n>0\}} \partial_n\left(
\frac{(2-\gamma)^{\frac\gamma{2-\gamma}}}{2^{\frac{\gamma-1}{2-\gamma}}}
x_n^{\frac\gamma{2-\gamma}}
\,\phi(x)\right)\,dx=0,
\end{eqnarray*}
providing an example of weak solution\footnote{We stress that, at this level, the solution considered in~\eqref{PKS93} is a weak solution. The setting will be different in the forthcoming
Theorems~\ref{APPDE-TH-BIS}
and~\ref{0o-2r-rtegripoHSdYjmsd-003PsKAM-l}, where classical solutions will be
taken into account, without integrability assumptions (the function
in~\eqref{PKS93} will however appear in~\eqref{BIS-HAN-poss-22}).} of~\eqref{LEQ} with a vanishing point (actually, a vanishing region)
in the interior of the domain.

For this reason, equation~\eqref{LEQ} when~$\gamma\in(0,2)$
has been widely investigated in the context of free boundary problems
and it is indeed the main topic of a classical article by
H. W. Alt and D. Phillips, see~\cite{MR850615}.

\begin{figure}[h] 
\includegraphics[width=0.5\textwidth]{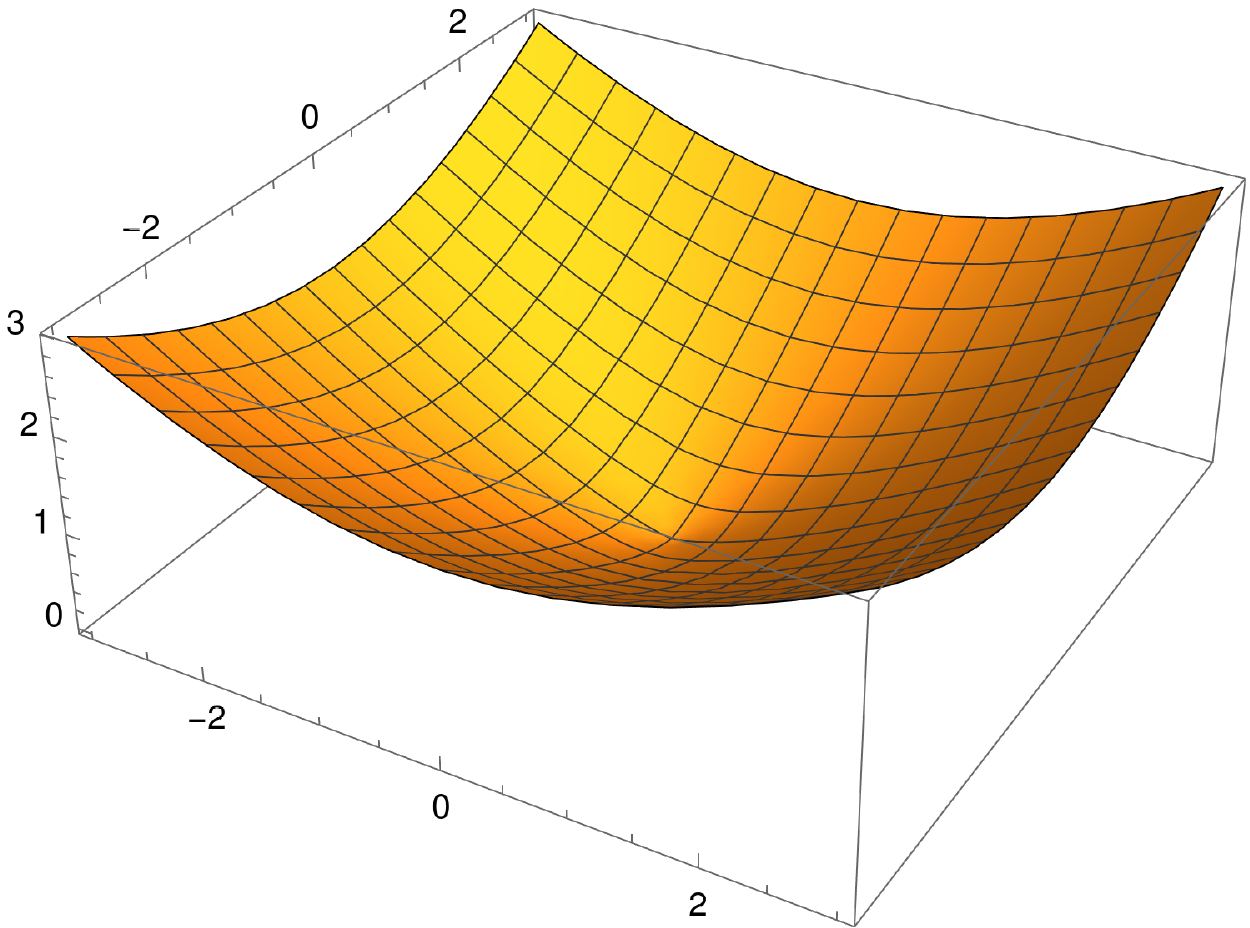} 
\caption{\footnotesize\sl The solution in \eqref{BIS-HAN-poss-1} with~$a:=4/3$.}
\label{FIG1}\end{figure}

{F}rom the point of view of applications,
equation~\eqref{LEQ} also models a reaction-diffusion problem
of a gas distribution in a porous catalyst pellet (see e.g.~\cite{zbMATH03494100}). 
To understand the regularity of the minimizers of
the associated energy functional and the way in which
the free boundary separates the zero set of the solution
from the positive region, one of the main tools relies on the blow-up
analysis of the problem, as well as
on the understanding of the corresponding homogeneous
solutions, see e.g. Sections~1.15 and~1.16 in~\cite{MR850615}
(see also Theorem~5.1 in~\cite{MR729195}
for the range~$\gamma\in(1,2)$).

The case~$\gamma=1$ in~\eqref{LEQF} corresponds to an obstacle problem
and is covered by
the classical work in~\cite{MR454350}.
Similarly, the case~$\gamma=0$ in~\eqref{LEQF} produces the seminal case studied in~\cite{MR618549}.
The case~$\gamma\in(0,1)$ has also been considered in~\cite{MR3980852}.

The case of negative
exponents~$\gamma$ appears to have been studied much less in the literature.
Once we have completed this paper,
the preprint~\cite{2022arXiv220307123D} has become available online, where the case~$\gamma\in(-2,0)$
has been taken into account (our perspective here is however quite different
from that in~\cite{2022arXiv220307123D}, since we do not focus our attention on the regularity of the local minimizers of the energy
functional but rather on classification results for global solutions, without energy constraints,
concentrating on the case of classical solutions).

The main goal of this paper is indeed to
consider all possible ranges of~$\gamma$, addressing particularly
the two-dimensional case.

Specifically, we focus on homogeneous solutions,
which play a special role in free boundary problems,
since this kind
of functions appear as limits of blow-ups and their classification
is thereby an essential ingredient towards a free boundary regularity theory.

A natural assumption for us, in view of the degree~$a$
of homogeneity
of the solution, is to consider the case in which~$u^{\frac1a}$
meets the zero set\footnote{For instance, in~\cite[Theorem~1.2]{MR3957397},
it is proved that when~$\gamma=1$,
under zero Dirichlet boundary data, the free boundary meets the fixed boundary in a~$C^1$ way and without density assumptions (and the result holds also in the fully nonlinear case).} in a
suitably regular fashion.
In this situation, as expected, one obtains
positive and rotationally invariant solutions, as well as
``one dimensional'' one phase solutions whose
positivity set is a halfplane. But, perhaps more surprisingly,
when~$a=1/2$ one also detects a ``resonance''
which produces new solutions whose positivity set
is a nontrivial cone (and even the union of different cones
whose opening is an acute angle).

\begin{figure}[h] 
\includegraphics[width=0.5\textwidth]{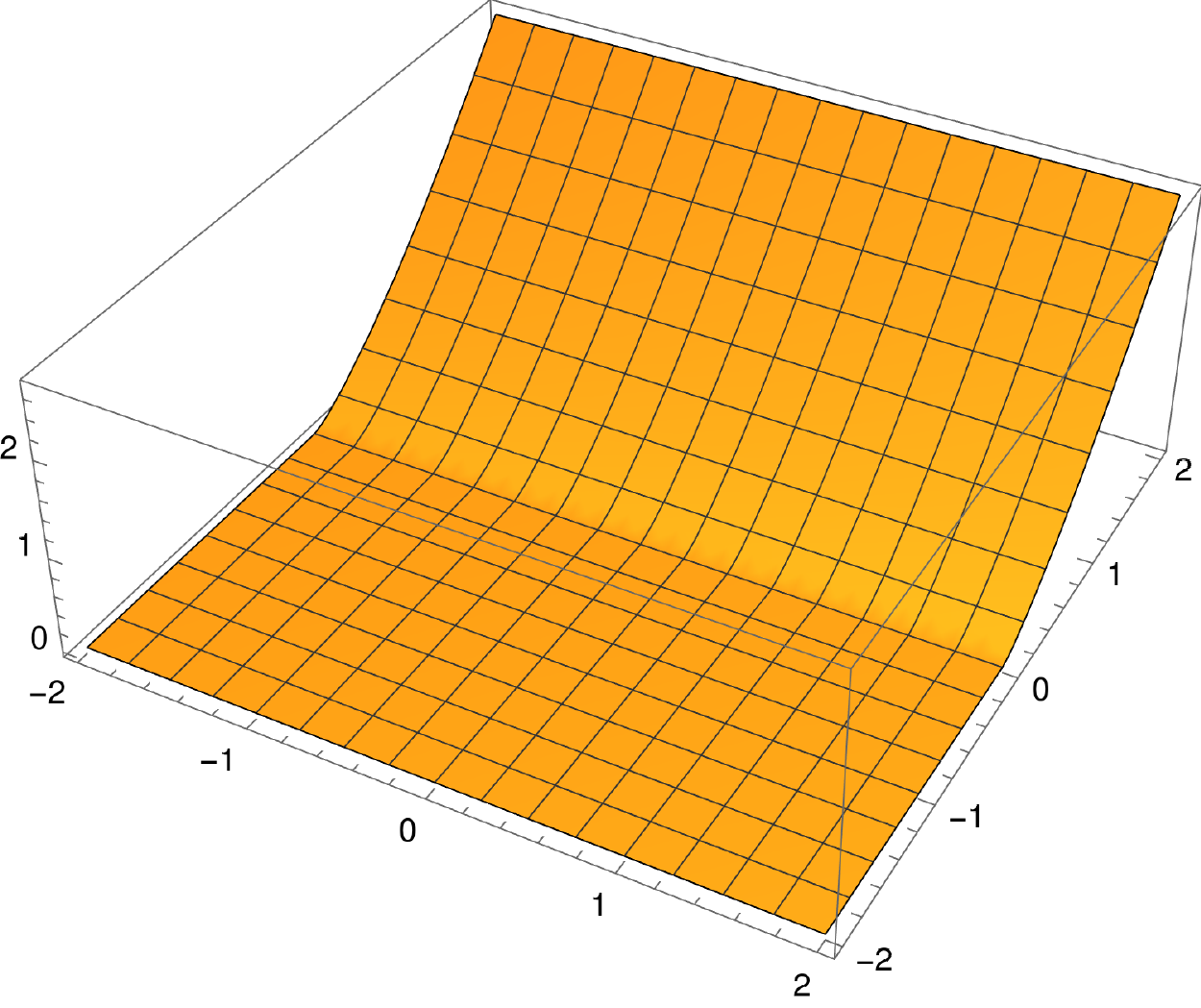} 
\caption{\footnotesize\sl The solution in \eqref{BIS-HAN-poss-22} with~$a:=4/3$.}
\label{FIG2}\end{figure}

The precise result that we have deals with classical solutions and is the following:

\begin{theorem}\label{APPDE-TH-BIS}
Let~$a>0$ and~$\gamma \ne0$. 
Assume that~$u\in C(\R^2)$ is a nontrivial, nonnegative, positively homogeneous solution of degree~$a$ of the equation
\begin{equation}\label{JAS:lkjhdf9tihyff8f8039875f}
\Delta u=\gamma u^{\gamma-1}\qquad{\mbox{in a connected component of }}\;\R^2\cap\{u>0\}.
\end{equation}

Then,
\begin{equation}
\label{BIS-EXP:LA} \gamma<2\qquad{\mbox{and}}\qquad
a=\frac{2}{2-\gamma}.
\end{equation}

{If~$a\ne\frac12$, suppose additionally\footnote{As customary, when~$\xi\in(0,+\infty)\setminus\N$, we can write~$\xi=\xi_1+\xi_2$, with~$\xi_1\in\N$ and~$\xi_2\in(0,1)$ and in this setting~$C^\xi$ is a short notation for~$C^{\xi_1,\xi_2}$,
that is having derivatives up to order~$\xi_1$, with the derivatives of order~$\xi_1$ satisfying
a H\"older condition with exponent~$\xi_2$.

{Notice that in condition~\eqref{LA23ureg}
a neighborhood of the origin is removed: the intuitive idea for it
is that, for a ``typical'' situation in the plane arising from homogeneous
solutions, the positivity set of the solution is given by some cone
and condition~\eqref{LA23ureg} aims at detecting the way in which the solution meets the free boundary at the regular points
(and not at the origin, where the free boundary may display a singularity).

We also point out that the equation~$\Delta u=\eta u^{\gamma-1}$ for any~$\eta\in\R$ such that~$\eta\gamma\in(0,+\infty)$
can be reduced to~\eqref{JAS:lkjhdf9tihyff8f8039875f} by setting~$v:=\left(\frac\gamma\eta\right)^{\frac1{2-\gamma}}u$.}} that, for each
connected component~$S$ of~$(B_2\setminus B_{1/2})\cap\{u>0\}$,
\begin{equation}\label{LA23ureg}
{\mbox{$u^{\frac1a}\in C^{\xi}(\overline{S})$ for some }}\xi>\begin{cases}
3-2a, &{\mbox{ if }} a\in(0,1),\\
\displaystyle \frac1a, &{\mbox{ if }} a>1.\end{cases}\end{equation}
Then, only the following
possible, non-exclusive, scenarios can happen}:\smallskip

\noindent{\bf[i]} we have
\begin{equation}\label{02}
\gamma\in(0,2)
\end{equation}
and
\begin{equation}\label{BIS-HAN-poss-1}
u(x)=C_a\,|x|^a,\qquad{\mbox{with }}\quad C_a:=\frac{(2(a-1))^{\frac{a}2}}{a^{\frac{3a}2}},
\end{equation}
{\bf[ii]}
up to a rotation,
\begin{equation}\label{BIS-HAN-poss-22}
u(x)=\frac{2^{\frac{a}2}}{a^a}\, (x_2)_+^a,\end{equation}
{\bf[iii]} up to a rotation,
\begin{equation}\label{BIS-HAN-poss-1-NUOVANs}
u(x)=\frac{2^{\frac{a}2}}{a^a}\,|x_2|^a,
\end{equation}
{\bf[iv]} the following situation occurs:
\begin{itemize}
\item $a=1/2$,
\item given~$c\in\R\setminus\{0\}$, up to a rotation and a reflection, the positivity set of~$u$
contains the cone
\begin{equation}\label{CONP} {\mathcal{C}}_c:=\Big\{ (r\cos\theta,r\sin\theta),\quad
r>0\quad{\mbox{and}}\quad\theta\in (0,T_c)
\Big\},\end{equation}
with
\begin{equation}\label{TICCI} T_c:=\begin{cases}
2\pi-2\arctan(1/c) &{\mbox{ when~$c>0$}},
\\-2\arctan(1/c)&{\mbox{
when~$c<0$,}}
\end{cases}\end{equation}
\item $u=0$ on~$\partial{\mathcal{C}}_c$,
\item for every~$x\in{\mathcal{C}}_c$,
\begin{equation}\label{AKJH709876543jhgc45} u(x)={2^{\frac{3}4}}\,
\sqrt{
x_2-cx_1+c|x|}.\end{equation}
\end{itemize}
\end{theorem}

\begin{figure}[h] 
\includegraphics[width=0.5\textwidth]{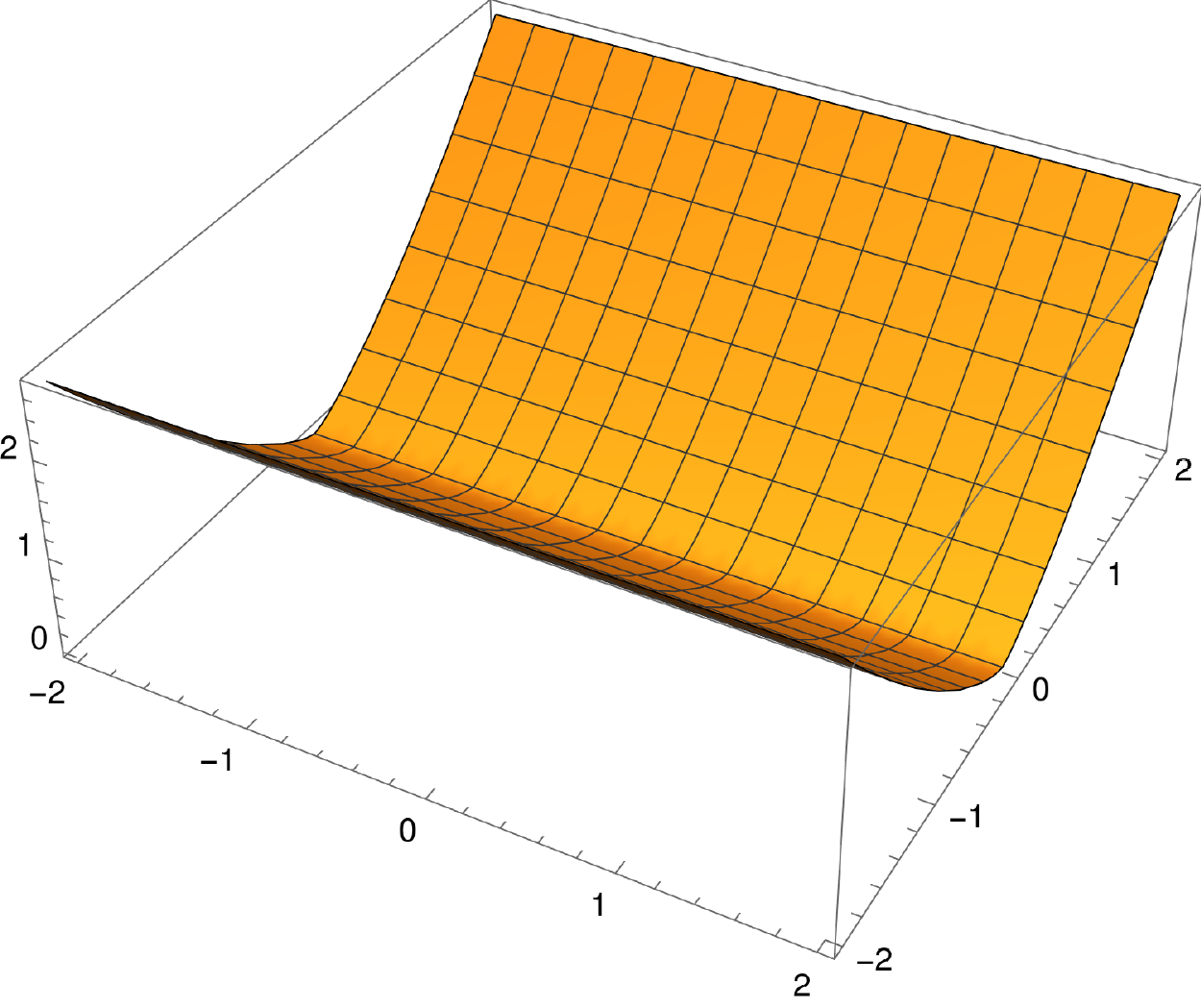} 
\caption{\footnotesize\sl The solution in \eqref{BIS-HAN-poss-1-NUOVANs} with~$a:=4/3$.}
\label{FIG2-c-c-c}\end{figure}

We stress that the scenarios described
in~[i], [ii], [iii] and~[iv] of
Theorem~\ref{APPDE-TH-BIS} are non-exclusive: namely,
when~$a=1/2$, the solution~$u$ can take any
of the forms in~\eqref{BIS-HAN-poss-22}, \eqref{BIS-HAN-poss-1-NUOVANs}
and~\eqref{AKJH709876543jhgc45}
(but not the form in~\eqref{BIS-HAN-poss-1}, since this requires~$\gamma>0$, that is~$a>1$).

Similarly, when~$\gamma\in(0,2)$, the solution can take
the expressions in~\eqref{BIS-HAN-poss-1}, \eqref{BIS-HAN-poss-22} and~\eqref{BIS-HAN-poss-1-NUOVANs}.

Another interesting feature of Theorem~\ref{APPDE-TH-BIS}
is that the ``degenerate'' case in which the free boundary
reduces to a single point, as described by~\eqref{BIS-HAN-poss-1},
can only occur when~$
\gamma\in(0,2)$, as detailed in~\eqref{02}.
Instead, the case~$\gamma<0$ only produces
a ``flat free boundary'', as given in~\eqref{BIS-HAN-poss-22},
with the only possible exception of~$\gamma=-2$,
in which a resonance can produce the situation described in~\eqref{AKJH709876543jhgc45}.

The solution in~\eqref{BIS-HAN-poss-22}
also coincides with that pointed out below~(2.3)
in~\cite{2022arXiv220307123D}.

Some of the solutions described in Theorem~\ref{APPDE-TH-BIS}
are depicted in Figures~\ref{FIG1}, \ref{FIG2}, \ref{FIG2-c-c-c} and~\ref{FIG3}.
See also Figure~\ref{FIG2-TAVBOL} for a table summarizing all these solutions.

\begin{figure}[h]  
\includegraphics[width=0.4\textwidth]{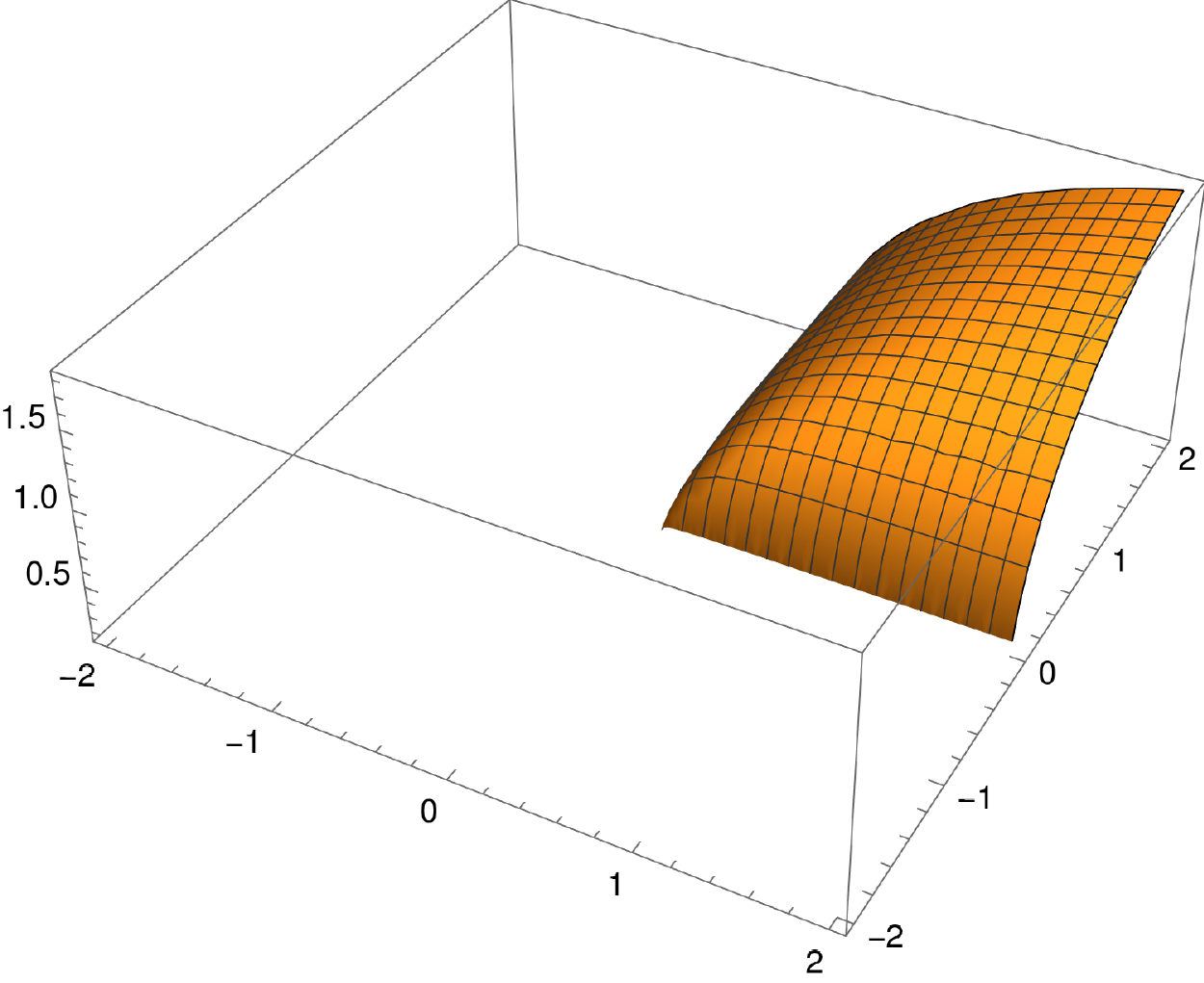}
\includegraphics[width=0.4\textwidth]{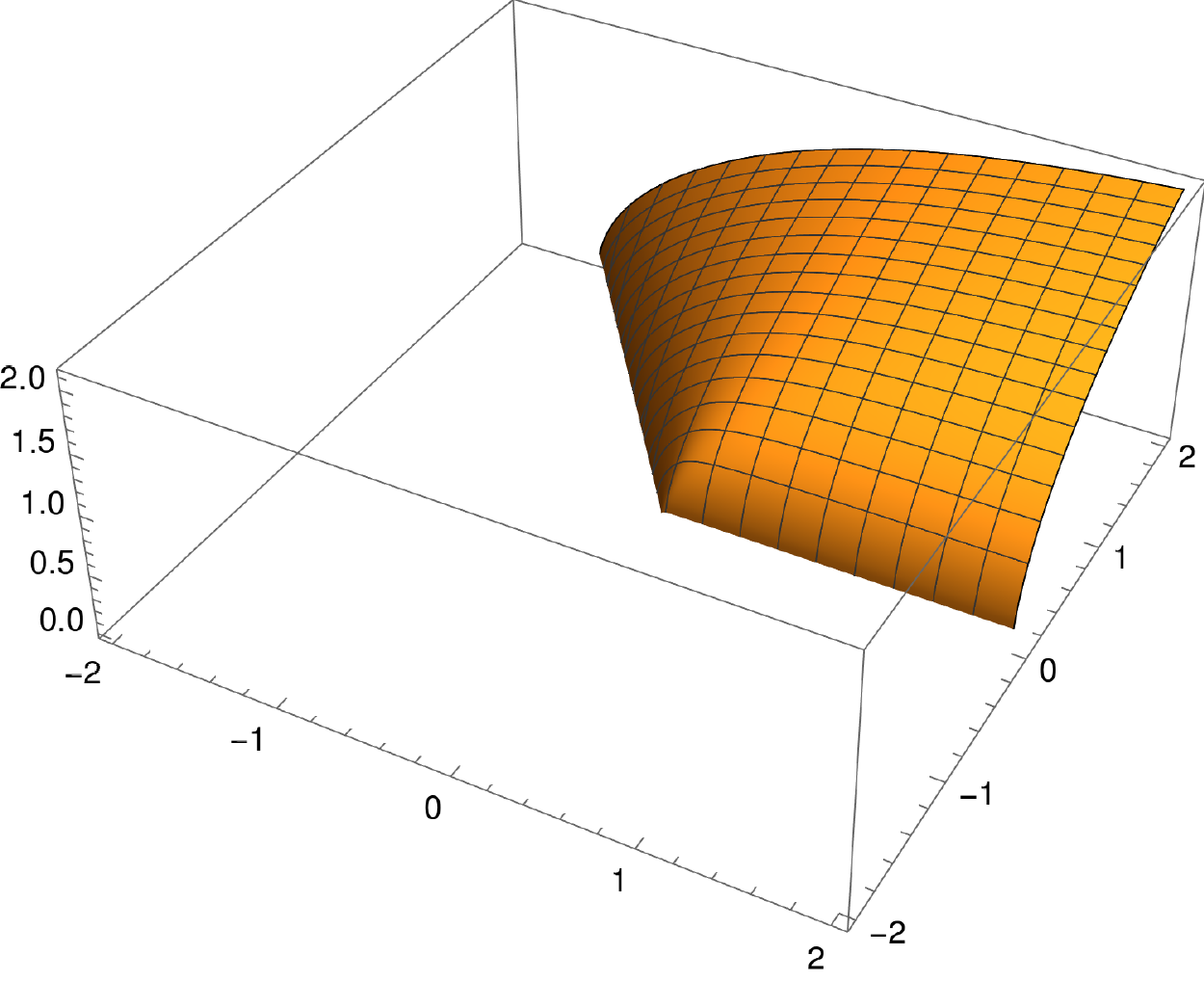}\\
\includegraphics[width=0.4\textwidth]{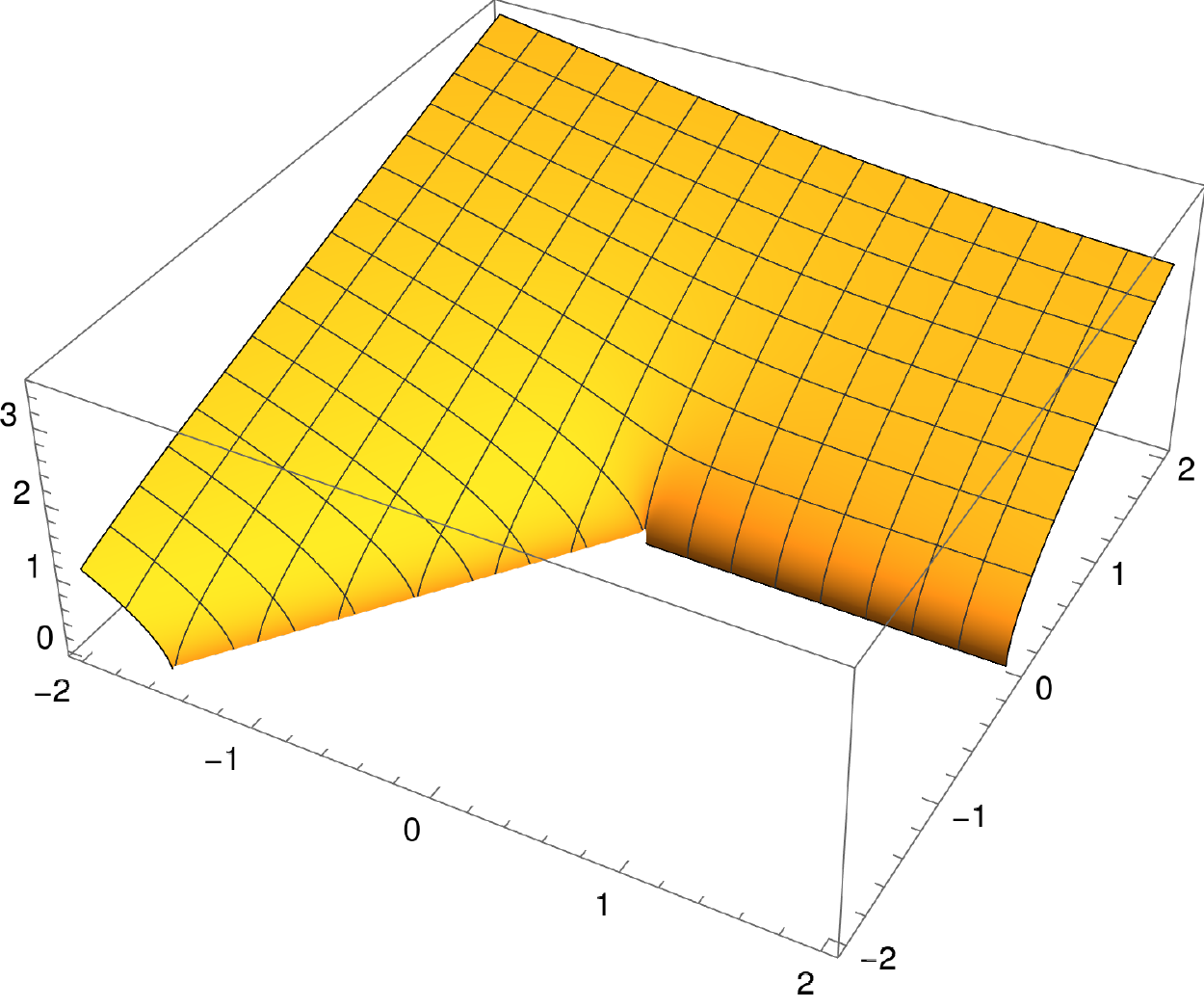}
\includegraphics[width=0.4\textwidth]{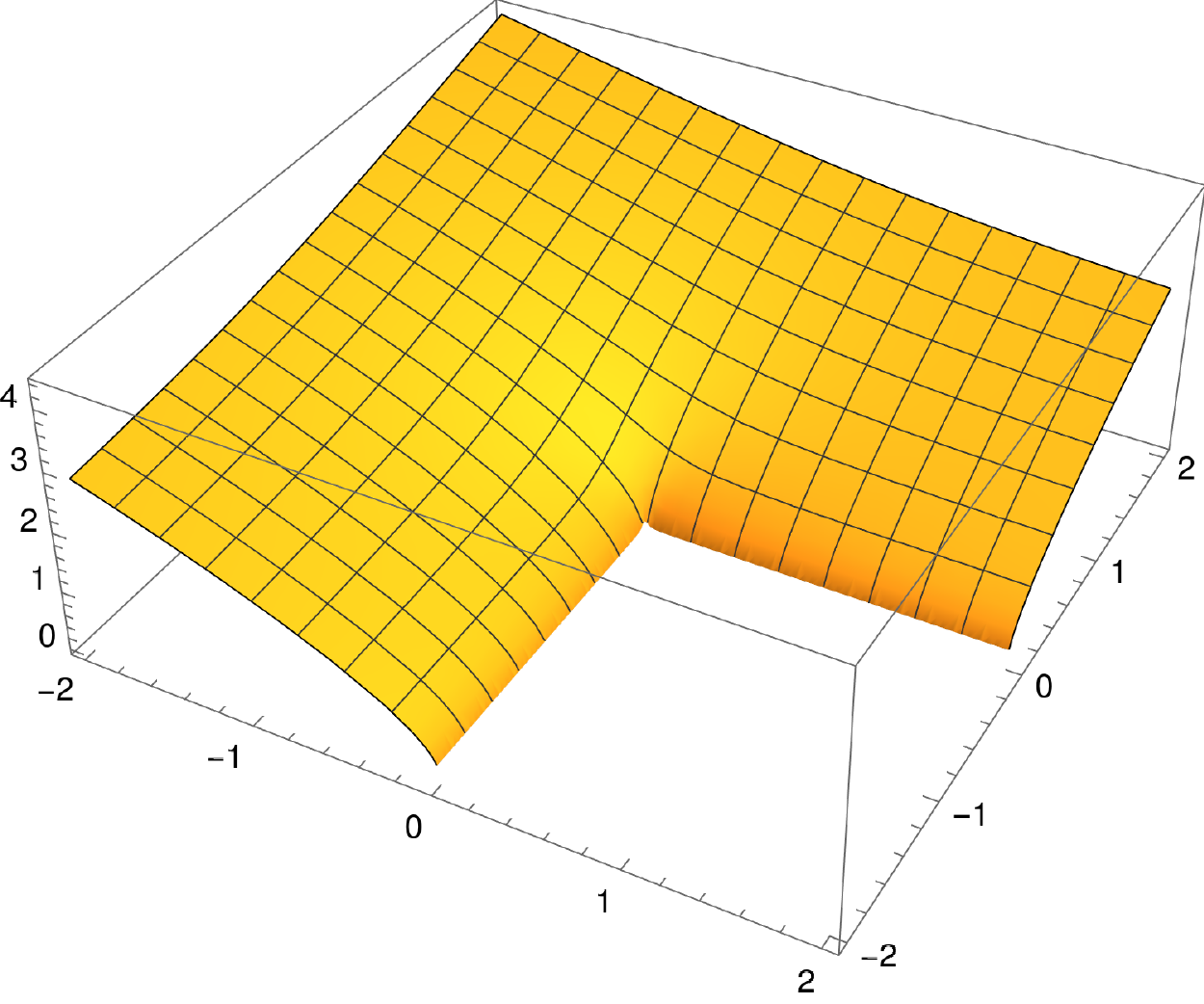}
\caption{\footnotesize\sl The solution in \eqref{AKJH709876543jhgc45} with~$c=-1,\;-1/2,\;1/2,\;1$.}
\label{FIG3}\end{figure}

In relation to~\eqref{TICCI}, we also remark that $T_c\in(\pi,2\pi)$
when~$c>0$, and~$T_c\in(0,\pi)$
when~$c<0$. In particular, the case~$c<0$ produces acute cones
in~\eqref{CONP}: in this scenario, the solutions
in~\eqref{AKJH709876543jhgc45} 
can be rotated and glued to form solutions with positive sets
in multi-flaps cones, see e.g. Figure~\ref{FIG4} (and, as a matter
of fact, these superpositions can be iterated, thus producing
also solutions whose positive sets is a cone with countably many
disjoint flaps).

\begin{figure}
\begin{center}     
\begin{tabular}{ | p{3cm} | p{8cm} | p{5cm} |}     
\hline  { \hspace{0.4cm} Ranges of~$\gamma$}  & {\hspace{2.5cm}}{Solution $u$}  & \ \ Free boundary $\partial\{u>0\}$\\ 
\hline 
\hline     \hspace{.5cm} {\bf[i]}  $\gamma\in(0, 2)$ &    {$u(x)=\frac{(2(a-1))^{\frac{a}2}}{a^{\frac{3a}2}}\,|x|^a,$}  & $\{0\}$\\ 
\hline     \hspace{.5cm} {\bf[ii]} $\gamma<2$&  $u(x)=\frac{2^{\frac{a}2}}{a^a}\, (x_2)_+^a,$ & $\{x\in \R^2, x_2=0\}$\\ 
\hline    \hspace{.5cm} {\bf[iii]} $\gamma<2$  & $u(x)=\frac{2^{\frac{a}2}}{a^a}\,|x_2|^a,$&  $\{x\in \R^2, x_2=0\}$\\ 
\hline    \hspace{.5cm} {\bf[iv]} $\gamma=-2$&   $u(x)={2^{\frac{3}4}}\,
\sqrt{
x_2-cx_1+c|x|}, $  & $\{ re^{i\theta},
r>0,\ \theta\in (0,T_c)
\},$
with $T_c$ given in \eqref{TICCI}
\\ 
\hline 
\end{tabular} 
\end{center} \caption{\footnotesize\sl Parade of the solutions detected in Theorem~\ref{APPDE-TH-BIS}.}
\label{FIG2-TAVBOL}
\end{figure}

\begin{figure}[h] 
\includegraphics[width=0.5\textwidth]{z.pdf} 
\caption{\footnotesize\sl Superposition of the solutions in \eqref{AKJH709876543jhgc45} with~$c=-1/2$
and~$c=-2$.}
\label{FIG4}\end{figure}

We also stress that condition~\eqref{LA23ureg} cannot be removed, otherwise
a family of new solutions arises, as detailed in the following result (in which condition~\eqref{LA23ureg}
is not assumed):

\begin{theorem}\label{0o-2r-rtegripoHSdYjmsd-003PsKAM-l}
Let~$a>0$ and~$\gamma \ne0$.
Assume that~$u\in C(\R^2)$
is a nontrivial, nonnegative, positively homogeneous solution of degree~$a$ of the equation
\begin{equation*}
\Delta u=\gamma u^{\gamma-1}\qquad{\mbox{in a connected component of }}\;\R^2\cap\{u>0\}.
\end{equation*}

Then,
\begin{equation*}
\gamma<2\qquad{\mbox{and}}\qquad
a=\frac{2}{2-\gamma}.
\end{equation*}
Also either~$u$ is one of the solutions listed in Theorem~\ref{APPDE-TH-BIS}
or~$a\ne\frac12$ and, up to a rotation,
$$ u(r,\theta)=\frac{2^{\frac{a}{2}}}{a^a}\,r^a \,y^a(\theta),$$
where the function~$y$ is defined implicitly by
\begin{equation*}
\theta=\int_0^{y(\theta)} \frac{dY }{ \sqrt{1+m\,Y^{2(1-a)}-Y^2} },
\end{equation*}
for some~$m\in\R$, with~$m\ge0$ if~$a>1$.
\end{theorem}

A particular explicit solution of the family listed in Theorem~\ref{0o-2r-rtegripoHSdYjmsd-003PsKAM-l}
is given by
\begin{equation}\label{909090-0909}
u(x)= \frac{x_2^2 + 2 x_1 x_2}{2}.
\end{equation}
This is a solution of~$\Delta u=1$ which is positive in the cone~$\{x_2(x_2+2x_1)>0\}$,
corresponding to a solution of~$\Delta u=\gamma u^{\gamma-1}$ with~$\gamma=1$.
See Figure~\ref{FIG2-c-c-9090c} for a diagram of this function
(and the forthcoming Remark~\ref{UA:SOLUN-i}
for its explicit link to the family of solutions presented in Theorem~\ref{0o-2r-rtegripoHSdYjmsd-003PsKAM-l}).

The paper is organized as follows. In Section~\ref{SAAHE} we present a brief heuristic discussion
of the ODE analysis performed in this paper and on the difficulties related to the singularity of the
associated Cauchy problem. The rigorous analysis begins in Section~\ref{REDOE}, where we reduce the PDE problem to
a non-standard ODE problem. Besides a family of explicit solutions, the ODE analysis
will leverage a special function of improper integral type and its inverse:
these additional functions will be introduced and studied in Sections~\ref{a-spec-PJ}, \ref{KPLS-wpofepjobfkpowfh},
and~\ref{G:6scH}. In Section~\ref{SEC2} we present a series of tailored
results on ODEs which will lead to the proof of Theorems~\ref{APPDE-TH-BIS} and~\ref{0o-2r-rtegripoHSdYjmsd-003PsKAM-l}, as
given in Section~\ref{SEC3}.

Finally, in Section~\ref{JS:COMMRER}
we remark that the implicit solutions presented in Theorem~\ref{0o-2r-rtegripoHSdYjmsd-003PsKAM-l}, when extended
by zero outside their positivity cone, are {\em not}
weak solutions of~$\Delta u=\gamma u^{\gamma-1}\chi_{\{u>0\}}$.

\begin{figure}[h] 
\includegraphics[width=0.5\textwidth]{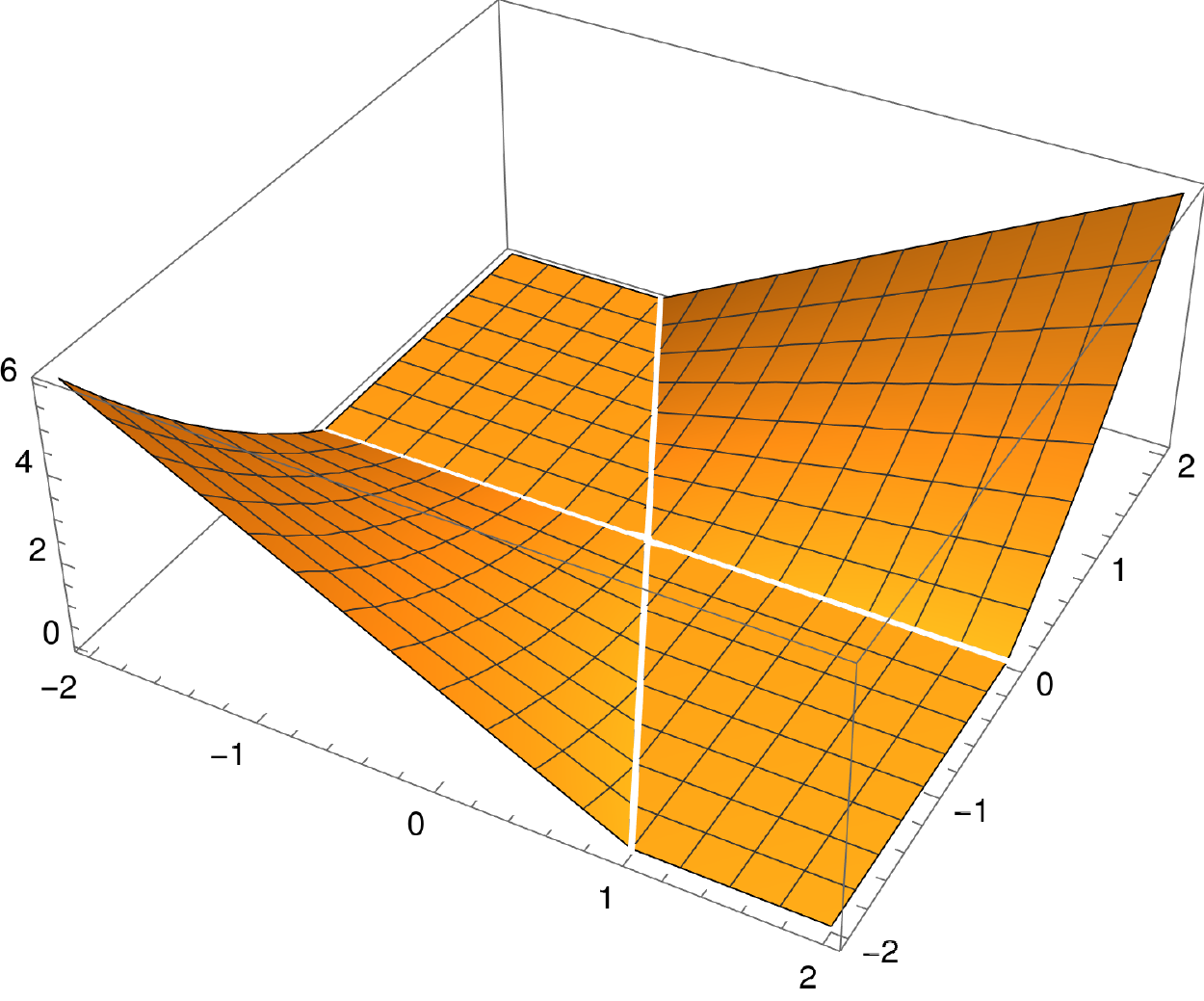} 
\caption{\footnotesize\sl The solution in \eqref{909090-0909}.}
\label{FIG2-c-c-9090c}\end{figure}

\section{A heuristic discussion}\label{SAAHE}

We give here a sketchy description of the ODE analysis related to our problem.
The classification of homogeneous solutions $u=r^ag(\theta)$ leads,
with the substitution~$y:=\frac{a}{\sqrt2} \,g^{\frac1a}$
to\footnote{{The choice of working with~$y$ instead of~$g$ presents assets and liabilities. On the one hand, the ODE for~$g$ is~$a^2g+g''=\frac{2(a-1)}{a}g^{\frac{a-2}{a}}$,
which is more standard than~\eqref{zkmd-01}.
On the other hand, the ODE in~\eqref{zkmd-01} has the advantage of placing the dependence on the exponent~$\gamma$ (i.e., on the parameter~$a$) only in the coefficients and of presenting useful algebraic properties in terms of factorization and reduction.

In a sense, the convenience of working with~$y$ instead of~$g$ is hinted by the ``one dimensional'' situation described by the solution in~\eqref{BIS-HAN-poss-22}, namely
$$ u(x)=\frac{2^{\frac{a}2}}{a^a}\, (x_2)_+^a
=\frac{2^{\frac{a}2}\,r^a}{a^a}\, (\sin\theta)_+^a.$$
In this case, in its positivity set~$g(\theta)$ would be~$\sin^a\theta$, while~$y(\theta)$ would have the simpler expression~$\sin\theta$.

The structural simplification in the one dimensional case is also our motivation to write the regularity assumption in~\eqref{LA23ureg} in terms of powers of~$\frac1a$.}} the 
ODE
\begin{equation}\label{zkmd-01}
y^2(\theta)+y (\theta)\,y''(\theta)+(a-1)\big(y^2(\theta)+(y'(\theta))^2-1\big)=0, 
\end{equation}
or
equivalently
\begin{equation*}
ay^2(\theta)+y (\theta)\,y''(\theta)+(a-1)(y'(\theta))^2+(1-a)=0.
\end{equation*}
This equation can be reduced to a first order ODE by substitution \begin{equation}\label{INVE}y'=u(y)\end{equation} arriving at
\begin{equation*}
ay^2(\theta)+yu'(y)y'+(a-1)u^2(y)+(1-a)=0, 
\end{equation*}
and using the new unknown function $u$ to substitute $y'$ in the last equation yields 
\begin{equation}
ay^2+(a-1)u^2+\frac y2(u^2)'+(1-a)=0.
\end{equation}
Finally, taking $U:=u^2(y)-1$ we have the linear first order ODE
$$
\frac y2 U'+(a-1)U+ay^2=0.
$$
The explicit solution of this ODE is given by 
\begin{equation}\label{zombi3}
(y')^2=1+my^{2(1-a)}-y^2, 
\end{equation}
where $m$ is the integration constant. 

Note that~\eqref{zombi3} gives an implicit relation of $\theta$ and $y$ and some extra care is needed to 
choose a branch of inverse function that produces the desired solution of our problem. Similarly, the change of independent variable from~$x$ to~$y$ utilized in~\eqref{INVE} needs to be justified, e.g. by showing that~$y'\ne0$ in the region of interest.
We remark also that, since the ODE in~\eqref{zkmd-01} is singular at the origin, it is not sufficient in our framework
just to ``exhibit'' a solution to complete a classification result, since in principle other solutions may arise due to a lack of uniqueness for a non-standard Cauchy problem.

Moreover, we observe that for one special case $a=2$ and~$m>0$ the implicit relation for $y$ takes form of an elliptic integral 
\begin{eqnarray*}
t=\int_0^{y(t)} \frac{Y\,dY }{ \sqrt{Y^2+m-Y^4} }
.\end{eqnarray*}
which can be solved explicitly, see Sections \ref {REDOE}, \ref{a-spec-PJ}
and~\ref{KPLS-wpofepjobfkpowfh}. However,  in general explicit representation of 
solutions are impossible: instead, the presence of unusual integral equations describing solutions in an implicit way
is actually the content of Theorem~\ref{0o-2r-rtegripoHSdYjmsd-003PsKAM-l}.

\section{Reduction to ODEs}\label{REDOE}

Here, we point out that, for a homogeneous function,
satisfying the partial differential equation in~\eqref{JAS:lkjhdf9tihyff8f8039875f} is equivalent to
have an appropriate power of the angular component satisfying a suitable ordinary differential equation.
The proof is a direct computation, though some care is needed since
the ordinary differential equation obtained is not a standard one.

\begin{lemma}\label{12341234}
Let~$a>0$ and~$\gamma\ne0$.
Let~$u:\R^2\setminus\{0\}\to\R$ be a homogeneous function of degree~$a$
expressed in polar coordinates~$(r,\theta)\in (0,+\infty)\times\S^1$ as
\begin{equation}\label{K:p0pp00} u(r,\theta)=r^a\,g(\theta).\end{equation}
Let~$S\subseteq\S^1$, assume that~$g>0$ in~$S$ and set
\begin{equation}\label{K:p0pp01}
y(\theta):=\frac{a}{\sqrt2} \,g^{\frac1a}(\theta).\end{equation}

If~$u$ is a solution of
\begin{equation}\label{P-SEDfgsottokmspoirA-0}
\Delta u=\gamma u^{\gamma-1}\qquad{\mbox{in }} (0,+\infty)\times S
\end{equation}
then
\begin{equation}
\label{00EXP:LA} \gamma<2,\qquad a=\frac{2}{2-\gamma}
\end{equation}
and~$y$ is a solution of
\begin{equation}\label{P-SEDfgsottokmspoirA-1}
y^2(\theta)+y (\theta)\,y''(\theta)+(a-1)\big(y^2(\theta)+(y'(\theta))^2-1\big)=0
\qquad{\mbox{for }} \theta\in  S.
\end{equation}

Conversely, if~\eqref{00EXP:LA} holds true and~$y$ is a solution of~\eqref{P-SEDfgsottokmspoirA-1}
then~$u$ is a solution of~\eqref{P-SEDfgsottokmspoirA-0}.
\end{lemma}

\begin{proof} We use the polar representation of the Laplace operator
\begin{equation*}
\begin{split}
\Delta u(x)\,&= \partial_r^2\big( r^a\,g(\theta)\big)+\frac1r \partial_r\big(r^a\,g(\theta)\big)
+\frac{1}{r^2}\partial_\theta^2\big(r^a\,g(\theta)\big)\\&=
a(a-1)r^{a-2} g(\theta)+ar^{a-2} g(\theta)+r^{a-2}g''(\theta)\\&=
a^2r^{a-2} g(\theta)+r^{a-2}g''(\theta)\\&=
a^2r^{a-2} \left(\frac{\sqrt2\,y(\theta)}{a}\right)^a+r^{a-2}\partial^2_\theta\left[ \left(\frac{\sqrt2\,y(\theta)}{a}\right)^a\right]\\&=
2^{\frac{a}{2}} a^{2-a} r^{a-2} y^a(\theta)+\frac{2^{\frac{a}2}r^{a-2}}{a^a}
\Big( a(a-1) y^{a-2}(\theta)(y'(\theta))^2+ay^{a-1}(\theta)y''(\theta)\Big)\\&=
2^{\frac{a}{2}} a^{1-a} r^{a-2} y^{a-2}(\theta)
\Big(ay^2(\theta)+(a-1) (y'(\theta))^2+y(\theta)y''(\theta)\Big).
\end{split}\end{equation*}
Therefore,
\begin{equation}\label{PJSNL0-prjeg23t4perjjX}\begin{split}&
\Delta u-\gamma u^{\gamma-1}\\ =\;&2^{\frac{a}{2}} a^{1-a} r^{a-2} y^{a-2}
\big(ay^2+(a-1) (y')^2+yy''\big)- \gamma (r^a\,g)^{\gamma-1}\\=\;&
r^{a-2}\left[
2^{\frac{a}{2}} a^{1-a} y^{a-2}
\big(ay^2+(a-1) (y')^2+yy''\big)- \gamma r^{a(\gamma-1)-a+2}
\left(\frac{\sqrt2}{a} y\right)^{a(\gamma-1)}\right]\\=\;&
r^{a-2}\left[
2^{\frac{a}{2}} a^{1-a} y^{a-2}
\big(ay^2+(a-1) (y')^2+yy''\big)-\frac{ 2^{\frac{a(\gamma-1)}2}\gamma}{a^{a(\gamma-1)}}
r^{a(\gamma-2)+2} y^{a(\gamma-1)}
\right].
\end{split}\end{equation}
For this reason, if~$u$ is a solution of~\eqref{P-SEDfgsottokmspoirA-0} then
\begin{equation}\label{xcsgKS:Md-12r3t4y}
2^{\frac{a}{2}} a^{1-a} y^{a-2}
\big(ay^2+(a-1) (y')^2+yy''\big)-\frac{ 2^{\frac{a(\gamma-1)}2}\gamma}{a^{a(\gamma-1)}}
r^{a(\gamma-2)+2} y^{a(\gamma-1)}=0.\end{equation}
Now, to prove~\eqref{00EXP:LA}, we suppose by contradiction that~$a\ne\frac{2}{2-\gamma}$
(note that if we reach a contradiction, then~\eqref{00EXP:LA} is established, since we assumed~$a>0$).
We thus write~\eqref{xcsgKS:Md-12r3t4y} as
\begin{equation}\label{xcsgKS:Md-12r3t4y-9}
2^{\frac{a}{2}} a^{1-a} y^{a-2}
\big(ay^2+(a-1) (y')^2+yy''\big)-\frac{ 2^{\frac{a(\gamma-1)}2}\gamma}{a^{a(\gamma-1)}}
r^{\delta} y^{a(\gamma-1)}=0\end{equation}
with~$\delta\ne0$. But it cannot be that~$\delta>0$, otherwise we would reach a contradiction by sending~$r\to+\infty$, and neither that~$\delta<0$, otherwise we would reach a contradiction by sending~$r\searrow0$. The proof of~\eqref{00EXP:LA} is thereby complete.

In the light of~\eqref{00EXP:LA}, equation~\eqref{xcsgKS:Md-12r3t4y} reduces to
\begin{equation*}
2^{\frac{a}{2}} a^{1-a} y^{a-2}
\big(ay^2+(a-1) (y')^2+yy''\big)- 2^{\frac{a}2 }a^{1-a}(a-1)
y^{a-2}=0,\end{equation*}
whence
\begin{equation*}
\big(ay^2+(a-1) (y')^2+yy''\big)- (a-1)=0,\end{equation*}
which is~\eqref{P-SEDfgsottokmspoirA-1}, as desired.

Now we assume that~$y$ satisfies~\eqref{P-SEDfgsottokmspoirA-1}
and~\eqref{00EXP:LA} holds true. Then, we infer from~\eqref{PJSNL0-prjeg23t4perjjX} that \begin{eqnarray*}&&
\frac{a^{a-1}\,r^{2-a}\,y^{2-a}}{2^{\frac{a}{2}}}\big(\Delta u-\gamma u^{\gamma-1}\big)
=
\big(ay^2+(a-1) (y')^2+yy''\big)-(a-1)=0,
\end{eqnarray*}
showing that~\eqref{P-SEDfgsottokmspoirA-0} holds true.
\end{proof}

\section{Two special functions $\psi(y)$ and $\Psi(y)$}\label{a-spec-PJ}

In this section, we study a special function built from an improper integral.

Let~$m\in\R$ and~$a\in(0,1)\cup(1,+\infty)$.
If~$a>1$ assume additionally that
\begin{equation}\label{ASSUNBGGIUNPEYBA}
m\ge0.
\end{equation}
For all~$y>0$ we let
\begin{equation}\label{xISPd92u0rj3-Xriyf2rbpo43-0}\psi(y):=1+m\,y^{2(1-a)}-y^2.\end{equation}
We observe that, by~\eqref{ASSUNBGGIUNPEYBA},
$$ \lim_{y\searrow0}\psi(y)=\begin{cases}
1 &{\mbox{ if }}a\in(0,1),\\
1&{\mbox{ if $a>1$ and~$m=0$,}}\\
+\infty&{\mbox{ if $a>1$ and~$m>0$.}} 
\end{cases}$$
As a result,
$$\lim_{y\searrow0}\psi(y)\ge1>-\infty=\lim_{y\to+\infty}\psi(y)$$
and therefore there exists a unique~$y_*>0$ such that
\begin{equation}\label{xISPd92u0rj3-Xriyf2rbpo43}
{\mbox{$\psi(y)>0$ for all~$y\in(0,y_*)$ and~$\psi(y_*)=0$.}}\end{equation}

\begin{lemma}
We have that
\begin{equation}\label{8709rowyndintha2}
\psi'(y_*)<0.
\end{equation}
\end{lemma}

\begin{proof} By~\eqref{xISPd92u0rj3-Xriyf2rbpo43}, for all small~$\e>0$,
$$ 0\le\frac{\psi(y_*-\e)}\e=\frac{\psi(y_*)-\psi'(y_*)\e+O(\e^2)}\e=-\psi'(y_*)+O(\e),$$
hence, sending~$\e\searrow0$, we find that~$\psi'(y_*)\le0$.

Consequently, to establish~\eqref{8709rowyndintha2}, it suffices to
check that
\begin{equation}\label{8709rowyndintha2-b}
\psi'(y_*)\ne0 .
\end{equation} To this end, suppose, by contradiction, that~$\psi'(y_*)=0$. Then,
\begin{equation*}
0=\psi'(y_*)=2(1-a)m\,y_*^{1-2a}-2y_*,\end{equation*}
whence~$y^{2a}_*=(1-a)m$.

This gives that
\begin{equation}\label{0iuojfel-COndfa-1}
(1-a)m>0\end{equation}
and
\begin{equation*}
y_*=\big((1-a)m\big)^{\frac1{2a}}.\end{equation*}

Accordingly, if~$\psi'(y_*)=0$, we deduce that
\begin{equation}\label{0iuojfel-COndfa-12}\begin{split}
&0=\psi(y_*)=\psi\left( \big((1-a)m\big)^{\frac1{2a}}\right)=
1+m\,\big((1-a)m\big)^{\frac{1-a}{a}}-\big((1-a)m\big)^{\frac1{a}}\\
&\qquad=1+\big((1-a)m\big)^{\frac1{a}}
\left( \frac1{1-a}-1\right)=1+\big((1-a)m\big)^{\frac1{a}}
\frac{a}{1-a}.
\end{split}\end{equation}
Necessarily, this gives that~$\frac{a}{1-a}<0$ and so~$a>1$.
Combined with~\eqref{0iuojfel-COndfa-1}, this establishes that~$m<0$, but this is
against our assumption in~\eqref{ASSUNBGGIUNPEYBA}.
The proof of~\eqref{8709rowyndintha2-b} is thereby complete.
\end{proof}

Now, for all~$y\in[0,y_*)$ we define
\begin{equation}\label{PKS-INmndsoitcrevES-0}
\Psi(y):=\int_0^y \frac{dY }{ \sqrt{\psi(Y)} }=\int_0^y \frac{dY }{ \sqrt{1+m\,Y^{2(1-a)}-Y^2} }.\end{equation}
Note that~$\Psi(0)=0$. Moreover,
\begin{equation}\label{PKS-INmndsoitcrevES}
{\mbox{$\Psi$ is a strictly increasing function of~$y\in[0,y_*)$}}\end{equation} and we may define
\begin{equation}\label{PKS-INmndsoitcrevES-1} t_*:=\lim_{y\nearrow y_*}\Psi(y)\in(0,+\infty].\end{equation}

Now we show that~$t_*$ is always finite, according to the next observation:

\begin{lemma}\label{8709rowyndintha3}
We have that~$t_*<+\infty$.
\end{lemma}

\begin{proof}
We let~$\e>0$, to be taken suitably small here below.
Then,
\begin{equation}\label{EXOpanTeystO09j}
\psi(y_*-\e)=-\psi'(y_*)\e + O(\e^2)=|\psi'(y_*)|\e + O(\e^2)\end{equation}
and therefore, for all~$\delta>\eta>0$ suitably small,
$$ \int_{y_*-\delta}^{y_*-\eta}\frac{dY }{ \sqrt{\psi(Y)} }=
\int^{\delta}_{\eta}\frac{d\e }{ \sqrt{\psi(y_*-\e)} }
=\int^{\delta}_{\eta}\frac{d\e }{ \sqrt{\e}\,\sqrt{|\psi'(y_*)|+O(\e)} }.$$
As a result, for small~$\delta$, using~\eqref{8709rowyndintha2} we find that
$$ \int_{y_*-\delta}^{y_*-\eta}\frac{dY }{ \sqrt{\psi(Y)} }\le\int^{\delta}_{\eta}\frac{d\e }{ \sqrt{\e}\,\sqrt{|\psi'(y_*)|/2} }\le
2\sqrt{\frac{2\delta}{|\psi'(y_*)|}}.
$$
Therefore, for all~$y\in [y_*-\delta,y_*)$,
$$  \Psi(y)=\Psi(y_*-\delta)+\int_{y^*-\delta}^y \frac{dY }{ \sqrt{\psi(Y)} }
\le \Psi(y_*-\delta)+2\sqrt{\frac{2\delta}{|\psi'(y_*)|}},$$
and therefore we can send~$y\nearrow y_*$ and obtain that~$t_*\le\Psi(y_*-\delta)+2\sqrt{\frac{2\delta}{|\psi'(y_*)|}}<+\infty$.
\end{proof}

\begin{figure}[h] 
\includegraphics[width=0.27\textwidth]{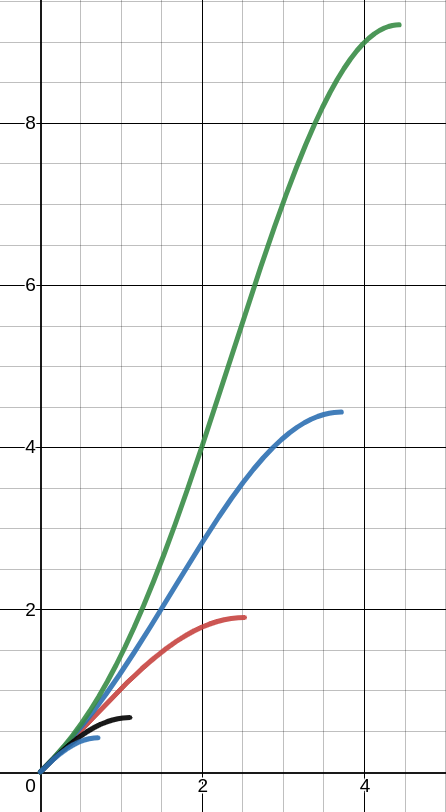} 
\caption{\footnotesize\sl Numerical plot of~$\Upsilon$ for~$t\in[0,t_*]$ when~$a=\frac14$ and~$m\in\{-3,-1,1,2,3\}$.}
\label{FYokmEXVG1PSI1}\end{figure}

\section{The inverse function $\Upsilon$ of $\Psi$}\label{KPLS-wpofepjobfkpowfh}

Now we aim at inverting the special function constructed in the previous section.
This method of implicitly inverting an integral equation is somewhat inspired to that used in the study of
cnoidal wave solutions to the Korteweg-de Vries equation, see e.g.~\cite{10.48550}.

The special function that we obtain reconstructs all the solutions of the ODE
presented in Section~\ref{REDOE} and therefore, in light of Lemma~\ref{12341234},
all the suitable powers of the angular components of the solutions of the PDE in~\eqref{JAS:lkjhdf9tihyff8f8039875f}.

Let us now present the analytical details of this construction.
In view of Lemma~\ref{8709rowyndintha3}, we can extend continuously~$\Psi$ at the point~$y_*$ by setting
\begin{equation}\label{JSNxatgbsdLSCItgbdy}\Psi(y_*):=t_*\in(0,+\infty).\end{equation} Thus, by~\eqref{PKS-INmndsoitcrevES}, we can define
the inverse function of~$\Psi:[0,y_*]\to[0,t_*]$, that is we denote by~$\Upsilon:[0,t_*]\to[0,y_*]$ the unique function such that~$\Psi(\Upsilon(t))=t$ (that is, $\Upsilon(t)$ is the unique solution~$y$ of the equation~$t=\Psi(y)$
and note that~$\Upsilon>0$ in~$(0,t_*]$).
We then evenly extend~$\Upsilon$ across~$t=t_*$ by setting, for all~$t\in(t_*,2t_*]$,
\begin{equation}\label{UJND-siym}
\Upsilon(t):=\Upsilon(2t_*-t).\end{equation}
In this way, $\Upsilon\in C([0,2t_*])$.

See Figures~\ref{FYokmEXVG1PSI1} and~\ref{FYokmEXVG1PSI2} for some numerical plots of~$\Upsilon$ for~$t\in[0,t_*]$.
The basic properties of this function are listed below:

\begin{prop} \label{3.3p}
We have that~$\Upsilon\in C^2((0,2t_*))$. Moreover, $\Upsilon(0)=0=\Upsilon(2t_*)$ and, for all~$t\in(0,2t_*)$, we have that
\begin{equation}\label{KMSeQUPSIKM-4} \Upsilon^2(t)+\Upsilon(t)\, \Upsilon''(t)+(a-1)\big(\Upsilon^2(t)+(\Upsilon'(t))^2-1\big)=0.\end{equation}
\end{prop}

\begin{proof} Since~$\Psi(\Upsilon(0))=0=\Psi(0)$, we obtain that~$\Upsilon(0)=0$.
As a result,~$\Upsilon(2t_*)=\Upsilon(2t_*-2t_*)=\Upsilon(0)=0$.

We also observe that~$\Upsilon\in C^2((0,t_*))$ and, for all~$t\in(0,t_*)$,
\begin{eqnarray*}
1=\frac{d}{dt}(t)=\frac{d}{dt}\big(\Psi(\Upsilon(t))\big)=\Psi'(\Upsilon(t))\,\Upsilon'(t)=
\frac{\Upsilon'(t)}{ \sqrt{1+m\,\Upsilon^{2(1-a)}(t)-\Upsilon^2(t)} },
\end{eqnarray*}
that is
\begin{equation}\label{LAOSndDRgsdd-pkm4R1} \Upsilon'(t)=\sqrt{1+m\,\Upsilon^{2(1-a)}(t)-\Upsilon^2(t)},\end{equation}
and thus
\begin{equation}\label{LAOSndDRgsdd-pkm4R2}\begin{split}&
\Upsilon''(t)=\frac{d}{dt}\left(\sqrt{1+m\,\Upsilon^{2(1-a)}(t)-\Upsilon^2(t)}\right)
=\frac{\Big( m(1-a)\,\Upsilon^{1-2a}(t)-\Upsilon(t)\Big)\,\Upsilon'(t)}{\sqrt{1+m\,\Upsilon^{2(1-a)}(t)-\Upsilon^2(t)}}\\&\qquad\qquad=
m(1-a)\,\Upsilon^{1-2a}(t)-\Upsilon(t).
\end{split}\end{equation}
By even symmetry, this also gives that~$\Upsilon\in C^2((0,t_*)\cup(t_*,2t_*))$.

Additionally, again by even symmetry,
$$ \Upsilon'(2t_*-t)=-\Upsilon'(t)\qquad{\mbox{and}}\qquad
\Upsilon''(2t_*-t)=\Upsilon''(t).$$
We also observe that~$\Psi(\Upsilon(t_*))=t_*=\Psi(y_*)$, whence~$\Upsilon(t_*)=y_*$ and therefore
$$ 1+m\,\Upsilon^{2(1-a)}(t_*)-\Upsilon^2(t_*)=1+m\,y_*^{2(1-a)}-y_*^2=\psi(y_*)=0.
$$
{F}rom the observations above, we infer that
\begin{eqnarray*}&&
\lim_{t\nearrow t_*}\Upsilon'(t)-\lim_{t\searrow t_*}\Upsilon'(t)=
\lim_{t\nearrow t_*}\Upsilon'(t)+\lim_{t\searrow t_*}\Upsilon'(2t_*-t)=2\lim_{t\nearrow t_*}\Upsilon'(t)\\&&\qquad=
2\sqrt{1+m\,\Upsilon^{2(1-a)}(t_*)-\Upsilon^2(t_*)}=0
\end{eqnarray*}
and therefore~$\Upsilon\in C^1((0,2t_*))$.

Also,
\begin{eqnarray*}&&
\lim_{t\nearrow t_*}\Upsilon''(t)-\lim_{t\searrow t_*}\Upsilon''(t)=
\lim_{t\nearrow t_*}\Upsilon''(t)-\lim_{t\searrow t_*}\Upsilon''(2t_*-t)=\lim_{t\nearrow t_*}\Upsilon''(t)-\lim_{t\nearrow t_*}\Upsilon''(t)=0
\end{eqnarray*}
and therefore~$\Upsilon\in C^2((0,2t_*))$, as desired.

It remains to check~\eqref{KMSeQUPSIKM-4}. For this, we observe that, if~$t\in(t_*,2t_*)$,
\begin{eqnarray*}&&
\Upsilon^2(2t_*-t)+\Upsilon(2t_*-t)\, \Upsilon''(2t_*-t)+(a-1)\big(\Upsilon^2(2t_*-t)+(\Upsilon'(2t_*-t))^2-1\big)\\&=&
\Upsilon^2(t)+\Upsilon(t)\, \Upsilon''(t)+(a-1)\big(\Upsilon^2(t)+(-\Upsilon'(t))^2-1\big)\\&=&
\Upsilon^2(t)+\Upsilon(t)\, \Upsilon''(t)+(a-1)\big(\Upsilon^2(t)+(\Upsilon'(t))^2-1\big),
\end{eqnarray*}
hence it suffices to check~\eqref{KMSeQUPSIKM-4} for~$t\in(0,t_*]$ (or, actually, for~$t\in(0,t_*)$
since the values at~$t_*$ can be reached by continuity).

To this end, in~$(0,t_*)$, we recall~\eqref{LAOSndDRgsdd-pkm4R1} and~\eqref{LAOSndDRgsdd-pkm4R2} and we compute that
\begin{eqnarray*}&&
\Upsilon^2+\Upsilon\, \Upsilon''+(a-1)\big(\Upsilon^2+(\Upsilon')^2-1\big)\\&&\qquad\qquad=
a\Upsilon^2+\Upsilon\, \big( m(1-a)\,\Upsilon^{1-2a}-\Upsilon\big)+(a-1)\big(
1+m\,\Upsilon^{2(1-a)}-\Upsilon^2-1\big)=0,
\end{eqnarray*}
completing the proof of the desired result.
\end{proof}

\begin{figure}[h] 
\includegraphics[width=0.33\textwidth]{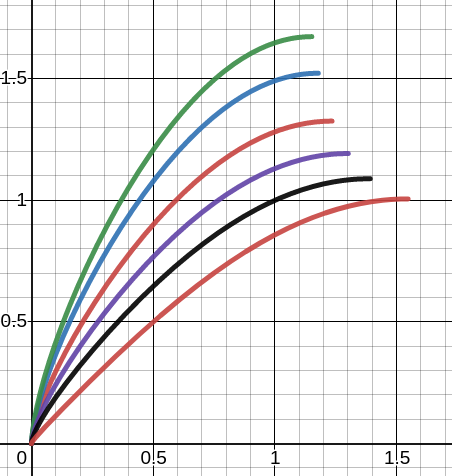} 
\caption{\footnotesize\sl Numerical plot of~$\Upsilon$ for~$t\in[0,t_*]$ when~$a=\frac32$ and~$m\in\{0.01, 0.2,0.5,1,2,3\}$.}
\label{FYokmEXVG1PSI2}\end{figure}

We now study the dependence of the above quantities with respect to the parameter~$m$
(considering~$a$ as given). For this, we use the notations~$\psi(y,m)$, $\Psi(y,m)$, $y_*(m)$, $t_*(m)$
and~$\Upsilon(t,m)$ to emphasize their dependence upon~$m$.

\begin{lemma}\label{athcrnothi}
We have that
\begin{equation}\label{0poi2rjetghm14567i8-2r3tgrtititiam}
y_*(0)=1\end{equation}
and
\begin{equation}\label{0p12543689oi2rjetghm-2r3tgrtititiam-2}t_*(0)=\frac\pi2.
\end{equation}
\end{lemma}

\begin{proof} By~\eqref{xISPd92u0rj3-Xriyf2rbpo43-0} and~\eqref{xISPd92u0rj3-Xriyf2rbpo43},
$$ 0=\psi(y_*(0),0)=1-(y_*(0))^2,$$
leading to~\eqref{0poi2rjetghm14567i8-2r3tgrtititiam}.

Moreover, by~\eqref{PKS-INmndsoitcrevES-0}, \eqref{PKS-INmndsoitcrevES-1} and~\eqref{0poi2rjetghm14567i8-2r3tgrtititiam},
\begin{eqnarray*}
t_*(0) =\lim_{y\nearrow y_*(0)}\Psi(y,0)=\lim_{y\nearrow 1}
\int_0^y \frac{dY }{ \sqrt{1-Y^2} }=\int_0^1 \frac{dY }{ \sqrt{1-Y^2} }=\frac\pi2,
\end{eqnarray*}
which demonstrates~\eqref{0p12543689oi2rjetghm-2r3tgrtititiam-2}.
\end{proof}

\begin{corollary}\label{athcrnothi2} There exists~${\mathcal{M}}\subseteq\R$ such that~${\mathcal{M}}\ne\varnothing$
for which if~$m\in {\mathcal{M}}$ the following holds true.

There exist~$t_*\in(0,\pi]$ and a function~$\Upsilon\in C^2((0,2t_*))\cap C([0,2t_*])$
such that~$\Upsilon(0)=0=\Upsilon(2t_*)$ and, for all~$t\in(0,2t_*)$, we have that~$\Upsilon(t)>0$ and
\begin{equation}\label{OSKM-23} \Upsilon^2(t)+\Upsilon(t)\, \Upsilon''(t)+(a-1)\big(\Upsilon^2(t)+(\Upsilon'(t))^2-1\big)=0.\end{equation}
Also, for all~$t\in[0,2t_*]$, we have that~$\Upsilon(2t_*-t)=\Upsilon(t)$ and
$$t=\int_0^{\Upsilon(t)} \frac{dY }{ \sqrt{1+m\,Y^{2(1-a)}-Y^2} }.$$
\end{corollary}

\begin{proof} The existence and basic properties of~$\Upsilon$ follow from Proposition~\ref{3.3p}.
The additional ingredient here is that we can find~${\mathcal{M}}\ne\varnothing$
such that when~$m\in{\mathcal{M}}$ it holds that~$t_*(m)\in(0,\pi]$,
which is warranted by Lemma~\ref{athcrnothi}.
\end{proof}

The importance of having that~$t_*\in(0,\pi]$ in Corollary~\ref{athcrnothi2} consists in being able to
use the function~$\Upsilon$ as a suitable power of the angular component of a solution of the PDE in~\eqref{JAS:lkjhdf9tihyff8f8039875f} (indeed, for this scope one wants that~$[0,2t_*]\subseteq[0,2\pi]$).

\section{Behavior of $\Upsilon$ near boundary point}\label{G:6scH}
Now we address the boundary regularity properties of the function~$\Upsilon$ introduced in Section~\ref{KPLS-wpofepjobfkpowfh}:

\begin{prop}\label{prop:basefier}
If~$m=0$, then~$\Upsilon(t)=\sin t$.

If instead~$m\ne0$, the following claims hold true.

\begin{itemize}
\item If~$a=\frac12$ then \begin{equation}\label{0ojlfw-COmdf-12}
\Upsilon(t)=\sin t+\frac{m}2(1-\cos t).
\end{equation}
\item If~$a\in(0,1)$ then~$\Upsilon\in C^1([0,2t_*])$,
\begin{equation}\label{JOSLN-21wekp2dwfme}
\Upsilon'(0)=1=-\Upsilon'(2t_*)
\end{equation}
and
\begin{equation}\label{JOSLN-21wekp2dwfme-2}
\lim_{t\searrow0}\frac{\Upsilon'(t)-\Upsilon'(0)}{t^{2(1-a)}}=\frac{m}2.
\end{equation}
\item More precisely, if~$a\in\left(0,\frac12\right)$ then
\begin{equation}\label{0ojlfw-COmdf-13}\begin{split}&
\Upsilon\in C^{2,1-2a}([0,2t_*])
\qquad{\mbox{with }}\quad\Upsilon''(0)=0\quad{\mbox{and }}\quad
\lim_{t\searrow0}\frac{\Upsilon''(t)-\Upsilon''(0)}{t^{1-2a}}
=(1-a)m,\\&{\mbox{ but
$\Upsilon\not\in C^{2,\xi}([0,2t_*])$ if~$\xi>1-2a$,}}
\end{split}
\end{equation}
while if~$a\in\left(\frac12,1\right)$ then
\begin{equation}\label{0ojlfw-COmdf-14}
{\mbox{$\Upsilon\in C^{1,2(1-a)}([0,2t_*])$, but
$\Upsilon\not\in C^{1,\xi}([0,2t_*])$ if~$\xi>2(1-a)$.}}
\end{equation}
\item If~$a>1$ then
\begin{equation}\label{0ojlfw-COmdf-15}
\begin{split}&
\lim_{t\searrow0}\frac{ \Upsilon(t)-\Upsilon(0)}{t^{\frac1a}}=a^{\frac1{a}}
{m}^{\frac1{2a}},\\&
{\mbox{$\Upsilon \in C^{\frac1a}([0,2t_*])$,
but $\Upsilon\not\in C^\xi([0,2t_*])$ when~$\xi>\frac1a$.}}
\end{split}\end{equation}
\end{itemize}
\end{prop}

\begin{proof}We focus on the regularity theory at~$t=0$, since the one at~$t=2t_*$ can be
inferred from that by symmetry, owing to~\eqref{UJND-siym}.

To this end, when~$m=0$ we have that~$t_*=\frac\pi2$, thanks to~\eqref{0p12543689oi2rjetghm-2r3tgrtititiam-2},
whence, for all~$t\in\left(0,\frac\pi2\right)$,
$$ t=\Psi(\Upsilon(t))=
\int_0^{\Upsilon(t)}\frac{dY }{ \sqrt{1-Y^2} }=\arcsin\Upsilon(t),
$$
therefore~$\Upsilon(t)=\sin t$, and this holds for all~$t\in(0,\pi)$, using the parity of~$\Psi$
across~$t=t_*=\frac\pi2$, as claimed.

Hence, we now suppose that~$m\ne0$. When~$a=\frac12$, we have that
$$ 0=\psi(y_*)=1+my_*-y_*^2$$
and therefore$$y_*=\frac{m + \sqrt{m^2+4}}{2}.$$
Consequently,
\begin{eqnarray*}&& t_*=\Psi(y_*)=\int_0^{y_*}\frac{dY }{ \sqrt{1+mY-Y^2} }
=\arctan\frac{m }{2}-\arctan\frac{m - 2y_*}{2\sqrt{1 + m y_* - y_*^2}}\\&&\qquad
=\arctan\frac{m }{2}+\arctan\frac{\sqrt{m^2+4}}{0^+}=\arctan\frac{m }{2}+\arctan(+\infty)=
\arctan\frac{m }{2}+\frac\pi2\in(0,\pi).
\end{eqnarray*}

Furthermore, for all~$t\in(0,t_*)$,
$$ t=\Psi(\Upsilon(t))=
\int_0^{\Upsilon(t)}\frac{dY }{ \sqrt{1+mY-Y^2} }=
\arctan\frac{m }{2}-\arctan\frac{m - 2 \Upsilon(t)}{2\sqrt{1 + m \Upsilon(t) - \Upsilon^2(t)}}
$$
and therefore
$$ \frac{m - 2 \Upsilon(t)}{2\sqrt{1 + m \Upsilon(t) - \Upsilon^2(t)}}=\tan T,$$
where $$
T:=\arctan\frac{m }{2}-t\in
\left(\arctan\frac{m }{2}-t_*,\arctan\frac{m }{2} \right)=\left(-\frac{\pi }{2},\arctan\frac{m }{2} \right)\subseteq
\left(-\frac{\pi }{2}, \frac{\pi }{2} \right)
.$$
This gives that
$$ \frac{m^2 - 4m \Upsilon(t)+4\Upsilon^2(t)}{4\tan^2 T}=1 + m \Upsilon(t) - \Upsilon^2(t),$$
and, as a result,
\begin{equation}\label{pkjqodwldv-2r3otikgrjh}
\begin{split}
0\,&=\, \Upsilon^2(t) - m \Upsilon(t)+ \frac{m^2}4 \cos^2 T -\sin^2 T
\\&=\,\left(\Upsilon(t) -\frac{m}{2}\right)^2- \frac{m^2+4}4 \sin^2 T\\
&=\,\left(\Upsilon(t) -\frac{m}{2}\right)^2- \frac{(m \cos t - 2\sin t)^2}4\\
&=\,\left(\Upsilon(t) -\frac{m}{2}\right)^2- \left(\frac{m}2 \cos t - \sin t\right)^2
\\&=\,\left(\Upsilon(t) -\frac{m}{2}-\frac{m}2 \cos t +\sin t\right)
\left(\Upsilon(t) -\frac{m}{2}+\frac{m}2 \cos t -\sin t\right)\\&=\,\left(\Upsilon(t) -\frac{m}{2}(1+ \cos t) +\sin t\right)
\left(\Upsilon(t) -\frac{m}{2}(1- \cos t) -\sin t\right)
.\end{split}\end{equation}
This yields that~$\Upsilon$ has the form stated in~\eqref{0ojlfw-COmdf-12}.
To check this, let us argue by contradiction and suppose that~$\Upsilon$ does
not agree with the form in~\eqref{0ojlfw-COmdf-12} in some interval~$(t_1,t_2)$,
with~$0\le t_1<t_2\le t_*$, and let us suppose that such interval is as large as possible.
Then, in view of~\eqref{pkjqodwldv-2r3otikgrjh}, we know that, for every~$t\in(t_1,t_2)$,
$$ \Upsilon(t) =\frac{m}{2}(1+ \cos t) -\sin t$$
and in particular~$\Upsilon(t_1)=\frac{m}{2}(1+ \cos t_1) -\sin t_1$.

Hence, necessarily~$t_1>0$, otherwise we would have that
$$0=\Upsilon(0)=\frac{m}{2}(1+ \cos 0) -\sin 0=m,$$
against our assumptions.

This observation and~\eqref{pkjqodwldv-2r3otikgrjh} give that
$$ \Upsilon(t)=\begin{cases}\displaystyle
\frac{m}{2}(1- \cos t) +\sin t&{\mbox{ if }}t\in(0,t_1),
\\ \\ \displaystyle\frac{m}{2}(1+ \cos t) -\sin t &{\mbox{ if }}t\in(t_1,t_2).
\end{cases}$$
Since, by~\eqref{OSKM-23}, we know that~$\Upsilon$ is~$C^1$ in a neighborhood of~$t_1$, we thereby find that
\begin{eqnarray*}
\frac{m}{2}\sin t_1 +\cos t_1=
\lim_{t\nearrow t_1}\Upsilon'(t)=
\lim_{t\searrow t_1}\Upsilon'(t)=-\frac{m}{2}\sin t_1 -\cos t_1,
\end{eqnarray*}
leading to
$$ 0=\frac{m}{2}\sin t_1 +\cos t_1
=\frac{m}{2}\cos \left(t_1-\frac\pi2\right) -\sin\left(t_1-\frac\pi2\right)
$$
and thus~$t_1=\frac\pi2+\arctan\frac{m}2=t_*$, once again in contradiction with our assumptions.
This gives~\eqref{0ojlfw-COmdf-12}, as desired.

Now we point out that, for~$t\in(0,t_*)$,
$$ 1=\frac{d}{dt} t=\frac{d}{dt}\Big(\Psi(\Upsilon(t))\Big)=\frac{d}{dt}\left(
\int_0^{\Upsilon(t)}\frac{dY }{ \sqrt{1+m\,Y^{2(1-a)}-Y^2} }\right)=
\frac{\Upsilon'(t) }{ \sqrt{1+m\,\Upsilon^{2(1-a)}(t)-\Upsilon^2(t)} },
$$
whence
\begin{equation}\label{I0-p;lbfX-4hk}
\Upsilon'(t) = \sqrt{1+m\,\Upsilon^{2(1-a)}(t)-\Upsilon^2(t)} .
\end{equation}
Therefore, if~$a\in(0,1)$, it follows that~$\Upsilon'$ is continuous at~$t=0$
and~$\Upsilon'(0)=1$, giving~\eqref{JOSLN-21wekp2dwfme}.

As a result, $\Upsilon(t)=t+o(t)$ as~$t\searrow0$, whence, by~\eqref{I0-p;lbfX-4hk},
\begin{equation}\label{09j0k9-0-2rkgfmm-PKMS3d}
\begin{split}
\Upsilon'(t)\,& =\, \sqrt{1+m\,\big(t+o(t)\big)^{2(1-a)}-\big(t+o(t)\big)^2} \\&=\,
1+\frac{m\,t^{2(1-a)}\big(1+o(1)\big)^{2(1-a)}-t^2\big(1+o(1)\big)^2}{2},
\end{split}\end{equation}
leading to~\eqref{JOSLN-21wekp2dwfme-2}.

We also deduce from~\eqref{09j0k9-0-2rkgfmm-PKMS3d} that,
as~$t\searrow0$,
$$ \Upsilon'(t)=
1+\frac{m\,t^{2(1-a)}}{2}+o(t^{2(1-a)})
$$
and thus
$$ \big(\Upsilon'(t)\big)^2=
1+ m\,t^{2(1-a)}+o(t^{2(1-a)})
.$$
As a consequence, by~\eqref{KMSeQUPSIKM-4},
\begin{equation}\label{0uwofehdvfb-23r4tyt}\begin{split}
0\,&=\frac{\Upsilon^2(t)+\Upsilon(t)\, \Upsilon''(t)+(a-1)\big(\Upsilon^2(t)+(\Upsilon'(t))^2-1\big)}t\\&=
\big(1+o(1)\big)\Upsilon''(t)+(a-1)\big(m\,t^{1-2a}+o(t^{1-2a })\big)+O(t)\\
&=\big(1+o(1)\big)
\, \Upsilon''(t)+(a-1)m\,t^{1-2a}+o(t^{1-2a}).\end{split}
\end{equation}

Let us now assume that~$a\in\left(0,\frac12\right)$. Then, the asymptotics in~\eqref{0uwofehdvfb-23r4tyt}
shows that~$\Upsilon\in C^2([0,2t_*])$, with~$\Upsilon''(0)=0$ and
\begin{equation}\label{0uwofehdvfb-23r4tyt-ru} \lim_{t\searrow0}\frac{\Upsilon''(t)-\Upsilon''(0)}{t^{1-2a}}
=(1-a)m.\end{equation}
Furthermore, employing~\eqref{KMSeQUPSIKM-4} for taking one more derivative, for small~$t>0$ we have that
\begin{eqnarray*} 
\Upsilon'''(t)&=&\frac{d}{dt}\big(\Upsilon''(t)\big)\\
&=&-\frac{d}{dt}\left(
a\Upsilon(t)+\frac{(a-1)\big((\Upsilon'(t))^2-1\big)}{\Upsilon(t)}\right)
\\&=&
-a\Upsilon'(t)+\frac{2(1-a) \Upsilon'(t)\Upsilon''(t)}{\Upsilon(t)}
+\frac{(a-1)\big((\Upsilon'(t))^2-1\big)\Upsilon'(t)}{\Upsilon^2(t)}\\&=&
-a\big(1+o(1)\big)+\frac{2(1-a) \big(1+o(1)\big)\big( (1-a)m\,t^{1-2a}+o(t^{1-2a})\big)}{t+o(t)}\\&&\qquad
+\frac{(a-1)\big(
m\,t^{2(1-a)}+o(t^{2(1-a)})
\big)\big(1+o(t)\big)}{t^2+o(t^2)}\\&=&-a+
(a-1)(2a-1)mt^{-2a}+o(t^{-2a}).
\end{eqnarray*}
For this reason,
if~$\hat{t}>0$ is sufficiently small and~$0<t_1<t_2<\hat{t}$, we infer that
\begin{eqnarray*}&&
|\Upsilon''(t_2)-\Upsilon''(t_1)|\le
\int_{t_1}^{t_2}|\Upsilon'''(t)|\,dt\le a\int_{t_1}^{t_2}dt+ 2(1-a)(1-2a)m
\int_{t_1}^{t_2}t^{-2a}\,dt\\&&\qquad=a(t_2-t_1)+
\frac{2(1-a)(1-2a)m}{1-2a}\big(t_2^{1-2a}-t_1^{1-2a}\big)\le C(t_2-t_1)^{1-2a},
\end{eqnarray*}
for some~$C>0$ depending only on~$a$ and~$m$, which shows that~$\Upsilon\in C^{2,1-2a}([0,2t_*])$.

Additionally, if~$\xi>1-2a$
$$ \lim_{t\searrow0}\frac{\Upsilon''(t)-\Upsilon''(0)}{t^\xi}=(1-a)m\lim_{t\searrow0}t^{1-2a-\xi}=+\infty,
$$
due to~\eqref{0uwofehdvfb-23r4tyt-ru},
hence~$\Upsilon\not\in C^{2,\xi}([0,2t_*])$.
The proof of~\eqref{0ojlfw-COmdf-13} is thereby complete.

Let us now deal with the case~$a\in\left(\frac12,1\right)$. In this situation, we deduce from~\eqref{0uwofehdvfb-23r4tyt} that
$$\Upsilon''(t)=(1-a)m\,t^{1-2a}+o(t^{1-2a})$$
and consequently, 
if~$\hat{t}>0$ is sufficiently small and~$0<t_1<t_2<\hat{t}$,
\begin{eqnarray*}&&|\Upsilon'(t_2)-\Upsilon'(t_1)|\le\int_{t_1}^{t_2}|\Upsilon''(t)|\,dt\le2(1-a)m
\int_{t_1}^{t_2}t^{1-2a}\,dt\\&&\qquad=
2(1-a)m\big(t^{2(1-a)}_2-t^{2(1-a)}_1\big)\le C(t_2-t_1)^{2(1-a)},
\end{eqnarray*}
showing that~$\Upsilon\in C^{1,2(1-a)}([0,2t_*])$.

However, if~$\xi>2(1-a)$,
$$ \lim_{t\searrow0}\frac{\Upsilon'(t)-\Upsilon'(0)}{t^{\xi}}=+\infty,$$
due to~\eqref{JOSLN-21wekp2dwfme-2}, hence~$\Upsilon\not\in C^{1,\xi}([0,2t_*])$.
We have therefore completed the proof of~\eqref{0ojlfw-COmdf-14}.

Now, we assume that~$a>1$. In this case, we have that
$$ \lim_{t\searrow0}\Upsilon^{1-a}(t)=+\infty,$$
therefore, for small~$t>0$, it is convenient to write~\eqref{I0-p;lbfX-4hk} in the form
\begin{equation}\label{pqjwlfe-o2ekr3efmg-20p3rtg-001} \Upsilon'(t) = 
\Upsilon^{1-a}(t)
\sqrt{m+\Upsilon^{2(a-1)}(t)-\Upsilon^{2a}(t)}=\big(\sqrt{m}+o(1)\big)\Upsilon^{1-a}(t)
\end{equation}
and accordingly
$$ \frac{d}{dt} \big(\Upsilon^a(t)\big)=a \Upsilon^{a-1}(t)\Upsilon'(t)=a
\sqrt{m}+o(1).$$
This entails that
$$ \Upsilon^a(t)=a\sqrt{m}\,t+o(t)$$
and therefore
\begin{equation}\label{pqjwlfe-o2ekr3efmg-20p3rtg-002} \Upsilon(t)=\big(a\sqrt{m}\,t+o(t)\big)^{\frac1a}=
\big(a^{\frac1{a}}{m}^{\frac1{2a}}+o(1)\big)\,t^{\frac1a}.\end{equation}
This shows that~$\Upsilon\not\in C^\xi([0,2t_*])$ when~$\xi>\frac1a$.

In addition, in light of~\eqref{pqjwlfe-o2ekr3efmg-20p3rtg-001} and~\eqref{pqjwlfe-o2ekr3efmg-20p3rtg-002},
$$
\Upsilon'(t)=\big(\sqrt{m}+o(1)\big)
\big(a^{\frac1{a}}{m}^{\frac1{2a}}+o(1)\big)^{1-a}\,t^{\frac{1-a}a}=
\big(a^{\frac{1-a}a}{m}^{\frac1{2a}}+o(1)\big)\,t^{\frac{1-a}a}.
$$
Owing to this, if~$\hat{t}>0$ is sufficiently small and~$0<t_1<t_2<\hat{t}$,
\begin{eqnarray*}&&|\Upsilon(t_2)-\Upsilon(t_1)|\le\int_{t_1}^{t_2}|\Upsilon'(t)|\,dt\le
2 a^{\frac{1-a}a}{m}^{\frac1{2a}}\,\int_{t_1}^{t_2}t^{\frac{1-a}a}\,dt\\&&\qquad\qquad
=2 a^{\frac1{a}}{m}^{\frac1{2a}} \big( t^{\frac1a}_2-t^{\frac1a}_1\big)\le C(t_2-t_1)^{\frac1a},
\end{eqnarray*}
for some~$C>0$, which demonstrates that~$\Upsilon \in C^{\frac1a}([0,2t_*])$.
This ends the proof of~\eqref{0ojlfw-COmdf-15}.
\end{proof}

\section{ODE methods}\label{SEC2}

This section contains some bespoke results on solutions of ordinary differential equations
which rely on the preliminary work done in the previous sections and
will be used in Section~\ref{SEC3} to establish Theorems~\ref{APPDE-TH-BIS} and~\ref{0o-2r-rtegripoHSdYjmsd-003PsKAM-l}.

\begin{lemma}\label{SIG}
Let~$a>0$, with~$a\ne1$. 
Let~$T_0>0$ and~$y\in C([0,T_0])\cap C^2((0,T_0))$ be a solution of
\begin{equation}\label{C:PB}
\begin{cases}
y^2+y y''+(a-1)(y^2+(y')^2-1)=0,\\
y(0)=0,\\
y(t)>0 {\mbox{ for all }}t\in(0,T_0).
\end{cases}
\end{equation}

Then, either
\begin{equation}\label{SOL:P} y(t)=\begin{cases}
\sin t& {\mbox{ if }}a\ne1/2,\\
\sin t + c \,\big(1- \cos t\big) & {\mbox{ if }}a=1/2,
\end{cases}\end{equation}
with~$c\in\R$,
or~$y(t)$ is implicitly defined by the relation
\begin{equation}\label{OKMScoiJNSM45e63r8KAS-0}
t=\int_0^{y(t)} \frac{dY }{ \sqrt{1+m\,Y^{2(1-a)}-Y^2} },
\end{equation}
for some~$m\in\R$, with~$m\ge0$ if~$a>1$.
\end{lemma}

\begin{proof} If~$y$
has the form claimed in~\eqref{SOL:P},
then it solves~\eqref{C:PB} by a direct computation. Furthermore, if~$y$
is as in~\eqref{OKMScoiJNSM45e63r8KAS-0}, then it solves~\eqref{C:PB},
due to Proposition~\ref{3.3p}.

Hence, it remains to prove that if
$y$ solves~\eqref{C:PB}, then it is of the form claimed in either~\eqref{SOL:P} or~\eqref{OKMScoiJNSM45e63r8KAS-0}.
To establish this, we first observe that if~$\widetilde{y}\in C([0,T_0])\cap C^2((0,T_0))$ is a solution of~\eqref{C:PB},
then, by the uniqueness result for regular Cauchy problems, we deduce that 
\begin{equation}\label{987651324sdcvvaysPJ}
\begin{split}&
{\mbox{if~$y(t)=\widetilde{y}(t)$
for all~$t$ in an interval~$I\subsetneqq[0,T_0]$,}}\\&{\mbox{then~$y(t)= \widetilde{y}(t)$
for all~$t\in[0,T_0]$.}}
\end{split}\end{equation}

As a consequence, in light of~\eqref{987651324sdcvvaysPJ}, it is sufficient
to prove that if
$y$ solves~\eqref{C:PB}, then it is of the form claimed in either~\eqref{SOL:P} or~\eqref{OKMScoiJNSM45e63r8KAS-0}
for all~$t$ in a suitable interval.

To this end, we observe that
\begin{equation}\label{YNOTCONS}
{\mbox{$y$ cannot be constant in an open interval.}}
\end{equation}
Indeed, suppose by contradiction that~$y(t)=c_0$ for all~$t$
in an open interval~$I$. Then, by~\eqref{C:PB}, for all~$t\in I$,
$$ 0=c_0^2+0+(a-1)(c_0^2+0-1)=ac_0^2-a+1.
$$
In particular, necessarily~$a\ne0$, and then~$c_0^2=\frac{a-1}a\ne0$.
This says that~$y$ is equal to~$c_0\ne0$ in an open interval and we can therefore divide by~$y$ in the ordinary differential
equation in~\eqref{C:PB} and extend the solution. But then, using the initial value in~\eqref{C:PB}, we see that~$0=y(0)=c_0\ne0$,
which is a contradiction and~\eqref{YNOTCONS} is proved.

Now, we define
\begin{equation}\label{825ug455} w(t):=
y^2(t)+(y'(t))^2-1.\end{equation}
We first suppose that $w$ vanishes identically in an open interval~$I$.
In this case, for all~$t\in I$,
$$ (y'(t))^2=1-y^2(t).$$
Also, by~\eqref{YNOTCONS}, we can find an interval~$I'\subseteq I$ in which~$y'\ne0$.
Thus we conclude that, for every~$t\in I'$,
$$b y'(t)=\sqrt{1-y^2(t)},$$
with~$b\in\{-1,1\}$ and accordingly
$$ \frac{d}{dt} \Big(b\arcsin y(t)-t\Big)=\frac{by'(t)}{\sqrt{1-y^2(t)}}-1=0.$$
{F}rom this we arrive at
\begin{equation}\label{LS90o3t-439yi54-25654754Dpoj3}
{\mbox{$y(t)=\sin\frac{t}{b}+\bar{c}=
b\sin t+\bar{c}$ for all~$t\in I'$}},\end{equation} where~$\bar{c}\in\R$.
Consequently, by~\eqref{C:PB}, for all~$t\in I'$,
\begin{equation}\label{AKSM:-we}
\begin{split}
0\,&=(b\sin t+\bar{c})^2-b(b\sin t+\bar{c}) \sin t+(a-1)\big((b\sin t+\bar{c})^2+\cos^2 t-1\big)\\
&=\bar{c} (a\bar{ c}+b( 2 a-1) \sin t).
\end{split}\end{equation}
This gives that
\begin{equation}\label{kamms-:aPKM}
\bar{ c}=0.\end{equation}
Because, if not, we deduce from~\eqref{AKSM:-we}
that the real analytic function~$a\bar{ c}+b(2 a-1) \sin t$ vanishes for all~$t\in I'$
and so, by analytic continuation, for all~$t\in\R$. Hence, taking~$t\in\left\{0,\frac\pi2\right\}$,
$$ a\bar{c}=0\qquad{\mbox{and}}\qquad a\bar{c}+b(2a-1)=0,$$
yielding that~$a=\frac12$ and then~\eqref{kamms-:aPKM}, as desired.

In light of~\eqref{kamms-:aPKM}, we deduce that~\eqref{LS90o3t-439yi54-25654754Dpoj3}
boils down to~$y(t)=b\sin t$ for all~$t\in I'$. Actually, by~\eqref{987651324sdcvvaysPJ},
we have that~$y(t)=b\sin t$ for all~$t\in[0,T_0]$.
Now, if~$b=-1$, then we obtain a
contradiction with the assumption that~$y(t)>0$. Therefore, we conclude that~$b=1$ and
accordingly~$y(t)=\sin t$ for all~$t\in[0,T_0]$, which is of the form claimed in~\eqref{SOL:P}.

Thus, from now on, we can assume that
\begin{equation}\label{tbsinSfoSll}
{\mbox{$w$ cannot be identically zero in an open interval.}}
\end{equation}
In this setting, we recall~\eqref{YNOTCONS} and we deduce that~$y'$ cannot be identically zero
in an open interval.
Furthermore, since~$y$ is analytic in~$(0,T_0)$ (being the solution of an
analytic Cauchy problem, see e.g. page~124 in~\cite{MR1707333}), we have that~$y'$ is
analytic in~$(0,T_0)$ as well, and therefore the set~$\{y'=0\}$ cannot
have accumulation points in~$(0,T_0)$.

As a consequence of this observation, we have that
there exists~${\mathcal{I}}\subset\N$ such that
\begin{equation}\label{3v436547y658u76ijhgtruhsgferHHHHH}
\{y'\neq 0\}=\bigcup_{i\in{\mathcal{I}}} (\theta_{i+1},\theta_i),
\end{equation}
where~$\theta_0=T_0$ and~$\theta_{i+1}\in[0,\theta_i)$.

We now claim that
\begin{equation}\label{KJSNd kfwehibkCCVBSDf oq934p5irt-0} (y'(t))^2=1+m\,y^{2(1-a)}(t)-y^2(t)
\quad{\mbox{for all }}t\in (0,T_0),
\end{equation}
for some~$m\in\R\setminus\{0\}$.

To prove it, we observe that
\begin{equation} \label{650234cdr}
w'=2yy'+2y'y''=2y'(y+y'').\end{equation}
Hence, in every interval of the form~$(\theta_{i+1},\theta_i)$,
we can divide by~$2y'$ and find that
\begin{equation*} \frac{w'}{2y'}= y+y''.\end{equation*}
This and~\eqref{C:PB} 
give that
\begin{equation*} 0=
y^2+y y''+(a-1)w=
y(y+ y'')+(a-1)w=\frac{w'y}{2y'}+(a-1)w,
\end{equation*}
which produces
$$ \frac{d}{dt}(\log |w|)=\frac{w'}{w}=-2(a-1)\frac{y'}{y}
=2(1-a)\frac{d}{dt}(\log |y|).$$
As a result, for every~$\e_i\in(\theta_{i+1},\theta_i)$ such that~$w(\e_i)\ne0$ (whose existence is warranted by~\eqref{tbsinSfoSll}) we have that
$$ \log \frac{|w(t)|}{|w(\e_i)|}=
2(1-a) \log\frac{|y(t)|}{|y(\e_i)|}.$$
Hence, since~$y(t)>0$,
\begin{equation}\label{wqui} |w(t)|=|w(\e_i)|\,\left(\frac{|y(t)|}{|y(\e_i)|}\right)^{2(1-a)}=
\frac{|w(\e_i)|}{y^{2(1-a)}(\e_i)}\,y^{2(1-a)}(t)
.\end{equation}
Without loss of generality, we can assume that
\begin{equation}\label{tbsinSfoSll-2}
{\mbox{$w$ has a strict sign in the interval~$(\theta_{i+1},\theta_i)$,}}
\end{equation}
otherwise we can pick a sequence of points~$\e_k\in(\theta_{i+1},\theta_i)$
such that~$w(\e_k)\ne0$,
$\e_k\to\overline\e_i\in(\theta_{i+1},\theta_i)$ as~$k\to+\infty$
and~$w(\overline\e_i)=0$. This and~\eqref{wqui} give that, for every~$t\in(\theta_{i+1},\theta_i)$,
\begin{equation*} |w(t)|=\lim_{k\to+\infty}\frac{|w(\e_k)|}{y^{2(1-a)}(\e_k)}\,y^{2(1-a)}(t)=
\frac{|w(\overline\e_i)|}{y^{2(1-a)}(\overline\e_i)}\,y^{2(1-a)}(t)=0,
\end{equation*}
in contradiction with our statement in~\eqref{tbsinSfoSll}.
This proves~\eqref{tbsinSfoSll-2}.

Also, by~\eqref{825ug455}, \eqref{wqui} and~\eqref{tbsinSfoSll-2}, for every~$t\in(\theta_{i+1},\theta_i)$,
\begin{equation}\label{AKP:MS-2ekrfemg} y^2(t)+(y'(t))^2-1=w(t)=
\frac{w(\e_i)}{y^{2(1-a)}(\e_i)}\,y^{2(1-a)}(t)=m_i\,y^{2(1-a)}(t),
\end{equation}
where
$$m_i:=\frac{w(\e_i)}{y^{2(1-a)}(\e_i)}\in\R\setminus\{0\}.$$

We claim that
\begin{equation}\label{jdi4e7584v96809876435436-976503}
m_{i+1}=m_i\quad {\mbox{for all }} i\in{\mathcal{I}}.
\end{equation}
Indeed, since~\eqref{AKP:MS-2ekrfemg} holds true in~$(\theta_{i+1},\theta_i)$
with coefficient~$m_i$ and in~$(\theta_{i+2},\theta_{i+1})$ with coefficient~$m_{i+1}$, 
we have that
\begin{eqnarray*}
&& y^2(t)+(y'(t))^2-1-m_i\,y^{2(1-a)}(t)=0\qquad{\mbox{ for all }}\,t\in(\theta_{i+1},\theta_i)
\\{\mbox{and }}&&
y^2(t)+(y'(t))^2-1-m_{i+1}\,y^{2(1-a)}(t)=0\qquad{\mbox{ for all }}\,t\in(\theta_{i+2},\theta_{i+1}).
\end{eqnarray*}
For this reason,
\begin{eqnarray*}&& y^2(\theta_{i+1})+(y'(\theta_{i+1}))^2-1-m_i\,y^{2(1-a)}(\theta_{i+1})\\&&\qquad\qquad
= 0=y^2(\theta_{i+1})+(y'(\theta_{i+1}))^2-1-m_{i+1}\,y^{2(1-a)}(\theta_{i+1}),\end{eqnarray*}
which gives~\eqref{jdi4e7584v96809876435436-976503}.

As a consequence of~\eqref{jdi4e7584v96809876435436-976503}, we can set~$m:=m_i$,
recall~\eqref{3v436547y658u76ijhgtruhsgferHHHHH}
and obtain that the equation in~\eqref{KJSNd kfwehibkCCVBSDf oq934p5irt-0}
is satisfied in~$\{y'\neq0\}$, and thus in~$\overline{ \{y'\neq0\} }\cap (0,T_0)=(0,T_0)$.
This completes the proof of~\eqref{KJSNd kfwehibkCCVBSDf oq934p5irt-0}.

Now we claim that
\begin{equation}\label{UJSNlsd-w0eurgjhoeriwskbnoerkf3utjhyrt}\begin{split}&
{\mbox{if $a\in(0,1)$ then }}\lim_{t\searrow0} y'(t)=1,\\
&{\mbox{and if~$a>1$ then~$m>0$ and }}\lim_{t\searrow0} y'(t)=+\infty.
\end{split}
\end{equation}
To check this, let us first suppose that~$a\in(0,1)$. Then, by~\eqref{KJSNd kfwehibkCCVBSDf oq934p5irt-0},
$$ \lim_{t\searrow0}
(y'(t))^2=\lim_{t\searrow0} \Big( 1+m\,y^{2(1-a)}(t)-y^2(t)\Big)=1.$$
Since~$y$ is positive for small~$t$, this gives~\eqref{UJSNlsd-w0eurgjhoeriwskbnoerkf3utjhyrt}
in this case.

Let us now suppose that~$a>1$. Thus,
using~\eqref{KJSNd kfwehibkCCVBSDf oq934p5irt-0} we obtain that
\begin{equation}\label{plqjwdfeSST6ydfghytherglrll} \lim_{t\searrow0}
(y'(t))^2=\lim_{t\searrow0} \Big( 1+m\,y^{2(1-a)}(t)-y^2(t)\Big)=1+m\; \infty,\end{equation}
therefore, in this case, since the left hand side is nonnegative,
we have that~$m\in(0,+\infty)$.
Hence, we obtain from~\eqref{plqjwdfeSST6ydfghytherglrll} that
\begin{equation*}\lim_{t\searrow0}
(y'(t))^2=+ \infty\end{equation*}
and the claim in~\eqref{UJSNlsd-w0eurgjhoeriwskbnoerkf3utjhyrt}
follows since~$y$ is positive for small~$t$.

As a consequence of~\eqref{UJSNlsd-w0eurgjhoeriwskbnoerkf3utjhyrt} we obtain that,
if~$\eta>0$ is chosen appropriately small, then, for all~$t\in(0,\eta)$,
\begin{equation}\label{KJSNd kfwehibkCCVBSDf oq934p5irt}
y'(t)\in(0,+\infty].
\end{equation}

Thus, exploiting~\eqref{KJSNd kfwehibkCCVBSDf oq934p5irt-0} and~\eqref{KJSNd kfwehibkCCVBSDf oq934p5irt}, we find that, for every~$t\in(0,\eta)$,
\begin{equation}\label{KJSNd kfwehibkCCVBSDf oq934p5irt-1}
y'(t)=\sqrt{1+m\,y^{2(1-a)}(t)-y^2(t)}.
\end{equation}

Now we recall the function~$\Psi$ introduced in~\eqref{PKS-INmndsoitcrevES-0} and we claim that, for every~$t\in[0,\eta)$,
\begin{equation}\label{KJSNd kfwehibkCCVBSDf oq934p5irt-2}
\Psi(y(t))=t.
\end{equation}
Indeed,
\begin{eqnarray*}
\lim_{t\searrow0} \Big(\Psi(y(t))-t\Big)=
\lim_{r\searrow0}
\int_0^r \frac{dY }{ \sqrt{1+m\,Y^{2(1-a)}-Y^2} }=0
\end{eqnarray*}
and, in view of~\eqref{KJSNd kfwehibkCCVBSDf oq934p5irt-1},
\begin{eqnarray*}
\frac{d}{dt} \Big(\Psi(y(t))-t\Big)=\Psi'(y(t))\,y'(t)-1=
\frac{y'(t) }{ \sqrt{1+m\,y^{2(1-a)}(t)-y^2(t)} }-1=0.
\end{eqnarray*}
These observations establish~\eqref{KJSNd kfwehibkCCVBSDf oq934p5irt-2}, as desired.

{F}rom~\eqref{KJSNd kfwehibkCCVBSDf oq934p5irt-2}, we deduce that, for every~$t\in[0,\eta)$,
the solution~$y(t)$ must coincide with the inverse function~$\Upsilon(t)$ of~$\Psi$,
as detailed in~\eqref{JSNxatgbsdLSCItgbdy}. This gives that~$y$ is as in~\eqref{OKMScoiJNSM45e63r8KAS-0}. In particular, if~$a=1/2$, $y(t)$ is as in~\eqref{SOL:P},
thanks to~\eqref{0ojlfw-COmdf-12} in Proposition~\ref{prop:basefier}.
\end{proof}

\begin{figure}[h] 
\includegraphics[width=0.37\textwidth]{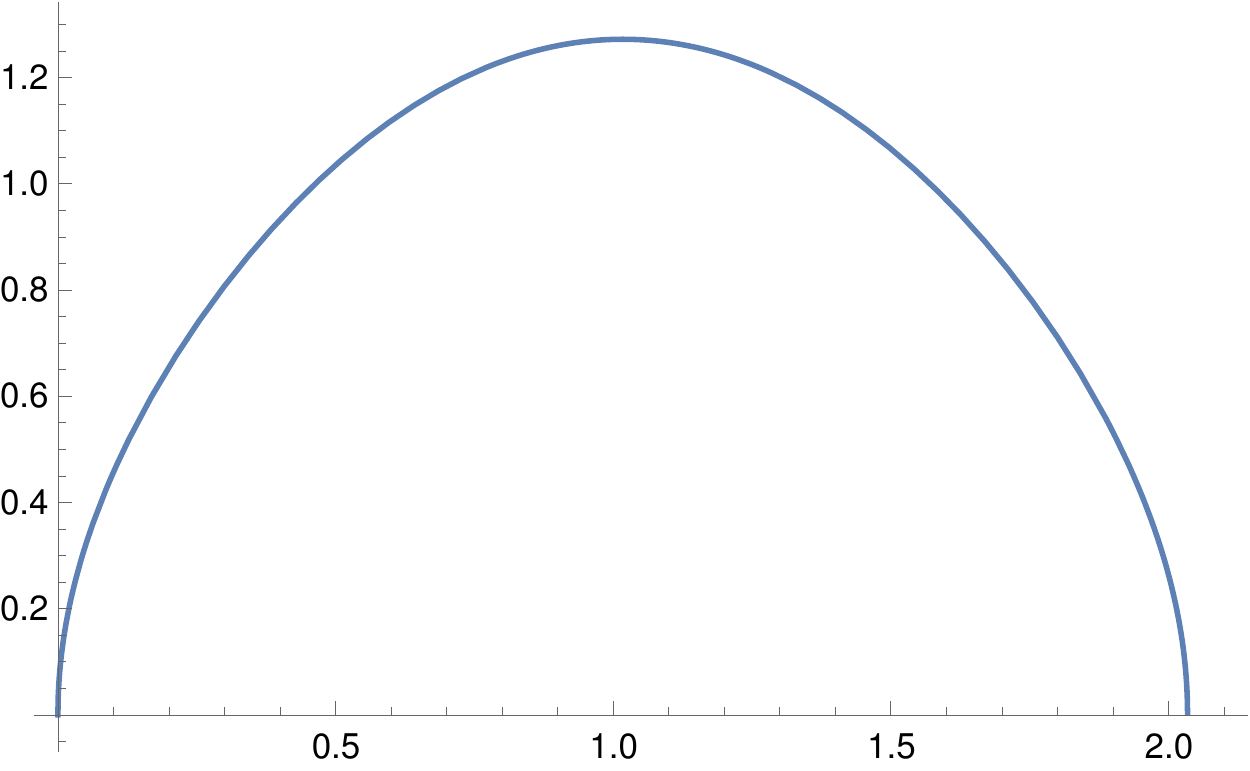} 
\caption{\footnotesize\sl The function in~\eqref{YokmEXV} for~$m:=1$.}
\label{FYokmEXVG1}\end{figure}

{ \begin{remark}\label{UA:SOLUN-i}
{\rm We stress that some explicit solutions
can be found among those presented in~\eqref{OKMScoiJNSM45e63r8KAS-0}. Namely, if~$a:=2$
and~$m>0$ then~\eqref{OKMScoiJNSM45e63r8KAS-0} reads
\begin{eqnarray*}
t=\int_0^{y(t)} \frac{Y\,dY }{ \sqrt{Y^2+m-Y^4} }
.\end{eqnarray*}
Since a primitive of~$\frac{Y}{ \sqrt{Y^2+m-Y^4} }$ is given by~$-\frac12\arctan\frac{1 - 2 Y^2}{2 \sqrt{m + Y^2 - Y^4}}$, we find that
\begin{eqnarray*}
t=\frac12\left(
\arctan\frac{1}{2 \sqrt{m}}
-\arctan\frac{1 - 2 y^2(t)}{2 \sqrt{m + y^2(t) - y^4(t)}}\right)
\end{eqnarray*}
and therefore
\begin{eqnarray*}
\frac{1 - 2 y^2(t)}{2 \sqrt{m + y^2(t) - y^4(t)}}=
\tan\left(\arctan\frac{1}{2 \sqrt{m}}-2t\right).
\end{eqnarray*}
Hence, using the trigonometric formula
$$ \tan(\alpha - \beta) = \frac{\tan \alpha - \tan\beta}{1 + \tan\alpha \tan\beta},$$
we obtain that
\begin{eqnarray*}
\frac{1 - 2 y^2(t)}{2 \sqrt{m + y^2(t) - y^4(t)}}=
\frac{\frac{1}{2 \sqrt{m}} - \tan(2t)}{1 + \frac{1}{2 \sqrt{m}} \tan(2t)},
\end{eqnarray*}
which gives that
$$ y^2(t)=\frac12 \Big(1 \pm\big(2 \sqrt{m} \sin(2 t) - \cos(2 t)\big)\Big).$$
Noticing that
$$ \lim_{t\searrow0}\frac{d}{dt}\big(y^2(t)\big)=\lim_{t\searrow0}
\pm \big(2 \sqrt{m} \cos(2 t) + \sin(2 t)\big)=
\pm 2 \sqrt{m} 
$$
we find that~$0\le y^2(t)=\pm 2 \sqrt{m} t+o(t)$ for small~$t$. This solves the sign ambiguity, leading to
$$ y^2(t)=\frac12 \Big(1 +\big(2 \sqrt{m} \sin(2 t) - \cos(2 t)\big)\Big)=
\frac{1-\cos(2t)}2+\sqrt{m} \sin(2 t)
$$
and therefore
\begin{equation}\label{YokmEXV}  y(t)=\sqrt{\frac{1-\cos(2t)}2+\sqrt{m} \sin(2 t)},\end{equation}
see Figure~\ref{FYokmEXVG1} for a diagram of this function when~$m:=1$.

Recalling Lemma~\ref{12341234}, we infer from this example that the function
$$ u=
\frac{r^2 \big(1 - \cos(2 \theta) + 2\sqrt{m}\sin(2 \theta) \big)}{4}=
\frac{x_2^2 + 2\sqrt{m} \,x_1 x_2}{2}
$$
is a solution of~$\Delta u=1$ (which can also be checked by a direct calculation). This observation is related to~\eqref{909090-0909}.}\end{remark}

The counterpart of
Lemma~\ref{SIG} for the non-singular equations
is given by the following result:

\begin{lemma}\label{6543mfghjH432Aba-234}
Let~$a>0$.
Assume that there exists a periodic
solution~$y\in C^2(\R)$ of the problem
\begin{equation}\label{22:C:PB}
\begin{cases}
y^2+y y''+(a-1)(y^2+(y')^2-1)=0,\\
\displaystyle\min_{\R} y>0.
\end{cases}
\end{equation}
Then, 
\begin{equation}\label{22IAKM345678iop000SNndc}
a>1
\end{equation}
and 
\begin{equation}\label{22IAKM345678iop000SNndc2d}
{\mbox{$y$ is constantly equal to }}\sqrt{\frac{a-1}a}.
\end{equation}
\end{lemma}

\begin{proof}
Let~$t_0\in\R$ be such that
$$ y(t_0)=\min_{[0,2\pi]} y>0.$$
Then, we have that~$y'(t_0)=0$ and~$y''(t_0)\ge0$.
This and the equation in~\eqref{22:C:PB} give that
\begin{equation}\label{767623409236r8teuf} \begin{split}
&0=y^2(t_0)+y (t_0)y''(t_0)+(a-1)(y^2(t_0)-1)\\&\qquad=
ay^2(t_0)+y (t_0)y''(t_0)-a+1>-a+1,\end{split}\end{equation}
which yields~\eqref{22IAKM345678iop000SNndc},
as desired.

It is also useful to remark that, in view of~\eqref{767623409236r8teuf},
$$ 0=
ay^2(t_0)+y (t_0)y''(t_0)-a+1\ge ay^2(t_0)-a+1,$$
and therefore
\begin{equation}\label{SKX2epswipp0} y(t_0)\le\sqrt{\frac{a-1}a}.\end{equation}
Similarly, if~$t_1$ is such that
\begin{equation}\label{SKX2epswipp0XCV} y(t_1)=\max_{[0,2\pi]} y>0,\end{equation}
we have that~$y'(t_1)=0$ and~$y''(t_1)\le0$,
whence the equation in~\eqref{22:C:PB} gives that
$$ 0=y^2(t_1)+y(t_1) y''(t_1)+(a-1)(y^2(t_1)-1)
\le ay^2(t_1)-a+1,$$
and accordingly
$$ y(t_1)\ge\sqrt{\frac{a-1}a}.$$

We claim that
\begin{equation}\label{GAHS:001} y(t_1)=\sqrt{\frac{a-1}a}.\end{equation}
For this, we argue by contradiction, supposing that
\begin{equation}\label{GAHS:002} y(t_1)>\sqrt{\frac{a-1}a}.\end{equation}
We define
$$ W(t):=1-y^2(t)-(y'(t))^2$$
and we observe that, in light of~\eqref{SKX2epswipp0},
\begin{equation}\label{SKX2epswipp}
W(t_0)= 1-y^2(t_0)\ge1- \frac{a-1}a=\frac{1}a>0.\end{equation}
Therefore, $W$ is strictly positive in
some interval~$I:=(t_0-\delta,t_0+\delta)$,
for a suitable~$\delta>0$. As a consequence, we can consider the logarithm of~$W$
in~$I$ and exploit
the equation in~\eqref{22:C:PB} to see that
\begin{eqnarray*}&& \frac{d}{dt}\log W=\frac{W'}{W}=\frac{-2y y'-2y'y''}{1-y^2-(y')^2}
=\frac{-2y'(y +y'')}{1-y^2-(y')^2}=\frac{-2y'(y^2 +yy'')}{y(1-y^2-(y')^2)}
\\&&\qquad=\frac{2(a-1)y'(y^2+(y')^2-1)}{y(1-y^2-(y')^2)}
=\frac{-2(a-1)y'}{y}=-2(a-1)\frac{d}{dt}\log y,
\end{eqnarray*}
and, as a result, for all~$t\in I$,
$$ \log\frac{W(t)}{W(t_0)}=-2(a-1)\log \frac{y(t)}{y(t_0)}=
\log\left( \frac{y(t)}{y(t_0)}\right)^{2(1-a)}.$$
Therefore, setting
\begin{equation}\label{SKX2epswipp1} \kappa:=\frac{W(t_0)}{(y(t_0))^{2(1-a)}},\end{equation}
we find that, for all~$t\in I$,
\begin{equation}\label{JA:askd92348r} 1-y^2(t)-(y'(t))^2=W(t)=\kappa (y(t))^{2(1-a)}.\end{equation}
We also remark that~$y$ is an analytic function, since it is a solution
of an analytic Cauchy problem (the sign condition in~\eqref{22:C:PB}
ensuring that the source term of the differential equation is non-singular,
after a division by~$y$), see e.g. page~124 in~\cite{MR1707333}.
Consequently, the relation in~\eqref{JA:askd92348r}
is globally valid, namely
\begin{equation} \label{01iwed--2e2u}
(y'(t))^2=1-y^2(t)-\kappa (y(t))^{2(1-a)}\qquad{\mbox{for all }}t\in\R.\end{equation}
Moreover, recalling~\eqref{SKX2epswipp0}, \eqref{SKX2epswipp}
and~\eqref{SKX2epswipp1},
$$ \kappa\ge\frac{1/a}{((a-1)/a)^{1-a}}=
\frac1a\,\left(\frac{a}{a-1}\right)^{1-a}.$$
For this reason and~\eqref{01iwed--2e2u}, we have that $$
0\le 1-y^2(t)-\kappa (y(t))^{2(1-a)}\le
1-y^2(t)-\frac1a\,\left(\frac{a}{a-1}\right)^{1-a}\, (y(t))^{2(1-a)}
\qquad{\mbox{for all }}t\in\R.$$
{F}rom this and~\eqref{GAHS:002}, we find that
$$ 0<1-\frac{a-1}a-\frac1a\,\left(\frac{a}{a-1}\right)^{1-a}\, \left(\frac{a-1}a\right)^{1-a}
=0.$$
This is a contradiction, and thus~\eqref{GAHS:001}
is established.

As a consequence of~\eqref{SKX2epswipp0XCV}
and~\eqref{GAHS:001}, we have that
$$ y(t_1)=\sqrt{\frac{a-1}a}\qquad{\mbox{and}}\qquad
y'(t_1)=0.$$
Since, by inspection, the function~$y_\star$ constantly equal to~$\sqrt{\frac{a-1}a}$
is also a solution of~\eqref{22:C:PB}, by the uniqueness result
of the standard Cauchy problem we infer that~$y(t)=y_\star(t)$
for every~$t\in\R$, and this proves the desired claim in~\eqref{22IAKM345678iop000SNndc2d}.
\end{proof}

\section{Proof of Theorems~\ref{APPDE-TH-BIS} and~\ref{0o-2r-rtegripoHSdYjmsd-003PsKAM-l}}\label{SEC3}

In light of Lemma~\ref{12341234},
we can express~$u$ in the polar form~$u(r,\theta)=r^a\,g(\theta)$
and we know that~$\gamma<2$, $a=\frac{2}{2-\gamma}$ and, setting~$y(\theta):=\frac{a}{\sqrt2} \,g^{\frac1a}(\theta)$, 
$$ y^2(\theta)+y (\theta)\,y''(\theta)+(a-1)\big(y^2(\theta)+(y'(\theta))^2-1\big)=0
\qquad{\mbox{for all }} \theta\in  S,$$
being~$S$ an open subset of~$\S^1$ (or simply of~$[0,2\pi]$
under periodicity assumptions).

Our goal is now to use the ODE analysis carried through in Section~\ref{SEC2}.
For this, to distinguish between the settings
in~\eqref{C:PB}
and~\eqref{22:C:PB}, we recall that~$y$ is nonnegative,
hence two cases may hold:
\begin{eqnarray}
\label{-0-103-00123-2}&&{\mbox{either }}\inf_{[0,2\pi]}y>0,\\
\label{-0-103-00123-1}&&{\mbox{or $y$ vanishes somewhere.}}
\end{eqnarray}
Assume first that~\eqref{-0-103-00123-2} holds true.
Then, $y$ is as in~\eqref{22:C:PB}, whence we can apply
Lemma~\ref{6543mfghjH432Aba-234} and infer that
\begin{equation}\label{XAcvXaxcecbepport}
a>1\end{equation}
and, for all~$\theta\in [0,2\pi]$,
$$\sqrt{\frac{a-1}a}=y(\theta)=
\frac{a}{\sqrt2}\,g^{\frac1a}(\theta).$$
This and~\eqref{K:p0pp00} give that
$$ u=\frac{(2(a-1))^{a/2}}{a^{3a/2}}\,r^a, $$
hence~\eqref{BIS-HAN-poss-1}
is established.

We also remark that
the function in~\eqref{BIS-HAN-poss-1} is indeed
a solution of~\eqref{JAS:lkjhdf9tihyff8f8039875f} since
\begin{eqnarray*}&&
\frac{(2(a-1))^{a/2}a(a-1)}{a^{3a/2}}\,r^{a-2}+
\frac{(2(a-1))^{a/2}a}{a^{3a/2}}\,r^{a-2}- \gamma 
\left(\frac{(2(a-1))^{a/2}}{a^{3a/2}}\,r^a\right)^{\gamma-1}
\\ &=&
\frac{(2(a-1))^{a/2}a^2}{a^{3a/2}}\,r^{a-2}- \frac{2(a-1)}{a} 
\left(\frac{(2(a-1))^{a/2}}{a^{3a/2}}\,r^a\right)^{(a-2)/a}\\ &=&\left(
\frac{(2(a-1))^{a/2}}{a^{(3a-4)/2}}- \frac{2(a-1)}{a} \,
\frac{(2(a-1))^{(a-2)/2}}{a^{3(a-2)/2}}\right)\,r^{a-2}\\&=&0.
\end{eqnarray*}
Finally, \eqref{02} follows from~\eqref{00EXP:LA}
and~\eqref{XAcvXaxcecbepport}.

Then, we can now focus on the case in which~\eqref{-0-103-00123-1}
is satisfied. Hence, up to a rotation, we can suppose that~$y>0$ in~$(0,T)$,
with~$y(0)=y(T)=0$ for some~$T\in(0,2\pi]$.
We then make use of Lemma~\ref{SIG}
(and note that~$a\ne1$, owing to~\eqref{00EXP:LA} and the assumption that~$\gamma\ne0$).
As a consequence,
we find that, for every~$\theta\in(0,T)$, either
\begin{equation}\label{LA:SK01o2rt} y(\theta)=\sin \theta+c(1-\cos \theta),\end{equation}
with~$c$ an arbitrary real constant when~$a=1/2$
and~$c=0$ when~$a\neq1/2$, or~$y(\theta)$ is implicitly defined by the relation
\begin{equation}\label{0o-2r-rtegripoHSdYjmsd-003PsKAM-lMK}
\theta=\int_0^{y(\theta)} \frac{dY }{ \sqrt{1+m\,Y^{2(1-a)}-Y^2} },
\end{equation}
for some~$m\in\R$, with~$m\ge0$ if~$a>1$.

The expression in~\eqref{0o-2r-rtegripoHSdYjmsd-003PsKAM-lMK}
is precisely the one proposed in Theorem~\ref{0o-2r-rtegripoHSdYjmsd-003PsKAM-l}.
We also stress that such an expression is excluded in Theorem~\ref{APPDE-TH-BIS},
thanks to assumption~\eqref{LA23ureg}. More precisely,
we know from~\eqref{0ojlfw-COmdf-13}, \eqref{0ojlfw-COmdf-14} and~\eqref{0ojlfw-COmdf-15}
that, if~\eqref{0o-2r-rtegripoHSdYjmsd-003PsKAM-lMK} holds true, then:
\begin{itemize}
\item if~$a\in\left(0,\frac12\right)$ and~$\xi>1-2a$, then~$y\not \in C^{2,\xi}$,
\item if~$a\in\left(\frac12,1\right)$ and~$\xi>2(1-a)$, then~$y\not \in C^{1,\xi}$,
\item if~$a>1$ and~$\xi>\frac1a$, then~$y\not\in C^\xi$,
\end{itemize}
and therefore assumption~\eqref{LA23ureg} excludes the appearance of solutions
described by~\eqref{0o-2r-rtegripoHSdYjmsd-003PsKAM-lMK} in Theorem~\ref{APPDE-TH-BIS}.

Therefore, it remains to check that~\eqref{LA:SK01o2rt} provides all the possible solutions
classified in the statement of Theorem~\ref{APPDE-TH-BIS}.

To this end, if~$a\ne1/2$, then~$y(\theta)=\sin\theta$
and~$T=\pi$. This gives that, for every~$x=(x_1,x_2)$
with~$x_2>0$,
$$ u=r^a g=\frac{2^{\frac{a}2}}{a^a} r^a y^a=
\frac{2^{\frac{a}2}}{a^a} (r\sin\theta)^a=\frac{2^{\frac{a}2}}{a^a} x_2^a.$$
This gives two possibilities:
\begin{eqnarray*}&&
{\mbox{ either }} u(x)=\frac{2^{\frac{a}2}}{a^a} (x_2)_+^a
\\&&
{\mbox{ or }} u(x)=\frac{2^{\frac{a}2}}{a^a} |x_2|^a
,\end{eqnarray*}
for all~$x\in\R^2$, therefore~\eqref{BIS-HAN-poss-22} and~\eqref{BIS-HAN-poss-1-NUOVANs}
are established in this case.

If instead~$a=1/2$, we have that, for every~$\theta\in(0,T)$,
$$ y(\theta)=\sin \theta+c(1-\cos \theta),$$
with~$c\in\R$, and the case~$c=0$ reduces to the previous
situation. Hence, we can suppose that~$c\ne0$ and we use
the formulae
$$ \cos \theta ={\frac  {1-\tau^{2}}{1+\tau^{2}}}
\qquad{\mbox{and}}\qquad \sin \theta =\frac {2\tau}{1+\tau^{2}}, \qquad{\mbox{where}}\qquad
\tau:=\tan{\frac {\theta }{2}} .$$
In this way, we have that
$$ y=\frac{2\tau(1+c\tau)}{1+\tau^2},$$
which is positive when~$\tau\in(-\infty,-1/c)\cup(0,+\infty)$
if~$c>0$, and when~$\tau\in(0,-1/c)$ when~$c<0$.

That is, $y(\theta)$
is positive when~$\theta\in(0,2\pi-2\arctan(1/c))$
if~$c>0$, and when~$\theta\in(0,-2\arctan(1/c))$ when~$c<0$.
This gives that~$T=2\pi-2\arctan(1/c)\in(\pi,2\pi)$
when~$c>0$, and that~$T=-2\arctan(1/c)\in(0,\pi)$
when~$c<0$.

Hence, in the cone~${\mathcal{C}}_c$ introduced in~\eqref{CONP}
we have that
$$ u=r^a g=\frac{2^{\frac{a}2}}{a^a}r^a y^a=\frac{
2^{\frac{a}2}}{a^a}r^a\big(
\sin \theta+c(1-\cos \theta)\big)^a=
\frac{2^{\frac{a}2}}{a^a}\big(
x_2-cx_1+c|x|\big)^a,
$$
and this is the setting described in~\eqref{AKJH709876543jhgc45}.

We stress that the function in~\eqref{BIS-HAN-poss-22}
satisfies~\eqref{LA23ureg} and also is a solution of~\eqref{JAS:lkjhdf9tihyff8f8039875f}, since, in this setting,
\begin{eqnarray*}&&\Delta u-\gamma u^{\gamma-1}
=\frac{
2^{\frac{a}2}a(a-1)}{a^a} x_2^{a-2}
-\gamma \left(\frac{2^{\frac{a}2} }{a^a}x_2^a\right)^{\gamma-1}
\\&&\qquad=\frac{
2^{\frac{a}2}a(a-1)}{a^a} x_2^{a-2}
-\frac{2a-2}a\,\frac{ 2^{\frac{a-2}2}}{a^{a-2}} x_2^{a-2}=0
\end{eqnarray*}
when~$x_2>0$,
thanks to~\eqref{BIS-EXP:LA}.

We also observe that the function in~\eqref{AKJH709876543jhgc45}
satisfies~\eqref{LA23ureg} and is a solution of~\eqref{JAS:lkjhdf9tihyff8f8039875f}, since
\begin{eqnarray*}&&\Delta u-\gamma u^{\gamma-1}\\&
=&\frac{2^{\frac{a}2}}{a^{a-1}}\,(x_2-cx_1+ c |x|)^{a-2}
\left[
(a-1) \left( \left(\frac{cx_1}{|x|}-c\right)^2 + \left(\frac{cx_2}{|x|}+1\right)^2\right)+  \frac{c}{|x|}
 (x_2-cx_1 + c |x|) 
\right]\\&&\qquad
-\gamma\left(\frac{2^{\frac{a}2}}{a^a}\big(
x_2-cx_1+c|x|\big)^a\right)^{\gamma-1}\\
&=&
\frac{2^{\frac{a}2}}{a^{a-1}}\,(x_2-cx_1+ c |x|)^{a-2}
\left[
(a-1) \left( 2c^2+1+\frac{2c}{|x|}(x_2-cx_1)\right)+  \frac{c}{|x|}
 (x_2-cx_1 + c |x|) 
\right]\\&&\qquad
-\frac{2a-2}{a}\,\frac{2^{\frac{a-2}2}}{a^{a-2}}\big(
x_2-cx_1+c|x|\big)^{a-2}\\
&=&
\frac{1}{2^{1/4}}\,(x_2-cx_1+ c |x|)^{-3/2}
\left[-\frac12 \left( 2c^2+1+\frac{2c}{|x|}(x_2-cx_1)\right)+  \frac{c}{|x|}
 (x_2-cx_1 + c |x|) 
\right]\\&&\qquad
+\frac{1}{2^{5/4}}\big(
x_2-cx_1+c|x|\big)^{-3/2}\\
&=&
\frac{1}{2^{1/4}}\,(x_2-cx_1+ c |x|)^{-3/2}
\left[-c^2-\frac12 +  c^2
\right]
+\frac{1}{2^{5/4}}\big(
x_2-cx_1+c|x|\big)^{-3/2}\\&=&0.
\end{eqnarray*}

\section{A comment about weak solutions}\label{JS:COMMRER}

We point out that none of the implicit solutions presented in Theorem~\ref{0o-2r-rtegripoHSdYjmsd-003PsKAM-l}, when extended
by zero outside their positivity cone,
is a weak solution of~$\Delta u=\gamma u^{\gamma-1}\chi_{\{u>0\}}$.
Indeed, suppose that one of these functions is a weak solution
and that its positivity cone is given by the set~$\{(r,\theta)\in\R\times(0,\varphi)\}$, for some~$\varphi\in(0,2\pi)$.

Consider a test function~$ \phi$ supported in a small ball~$B$ around~$e_1=(1,0)$.
Then,
$$ \int_{B} \nabla u(x) \cdot\nabla \phi(x)\,dx = -\gamma
\int_{B} u^{\gamma-1}(x)\chi_{\{u>0\}} (x)\phi(x)\,dx = -\gamma\int_{B^+} u^{\gamma-1}(x) \phi(x)\,dx,$$
where~$B^+:=B\cap\{x_2>0\}$.

But \begin{eqnarray*}&&\int_{B} \nabla u(x)\cdot \nabla \phi (x)\,dx= \int_{B\cap\{u>0\}} \nabla u(x)\cdot \nabla \phi(x)\,dx = \int_{B^+} \nabla u(x)\cdot \nabla \phi (x)\,dx\\&&\qquad= \int_{B^+} {\rm div}\big(\phi(x)\,
\nabla u(x)\big)\,dx -\int_{B^+}\Delta u(x)\phi(x)\,dx\\&&\qquad
= -\int_{H} \phi(x_1,0^+) \partial_2 u(x_1,0^+)\,dx_1-\gamma \int_{B^+}u^{\gamma-1}(x)\phi(x)\,dx
,\end{eqnarray*}
where~$ H:=B\cap\{x_2=0\}$.

Therefore $
\partial_2 u(x_1,0^+)=0$ along the $x_1$-axis.

But the implicit solutions, constructed in Theorem~\ref{0o-2r-rtegripoHSdYjmsd-003PsKAM-l}
do not satisfy this condition, since (up to multiplicative constants that we omit for simplicity):
\begin{itemize}\item
if $a\in(0,1)$, then $y(\theta)=\theta\,(1+o(1))$, due to \eqref{JOSLN-21wekp2dwfme},\item
if $a>1$, then $y(\theta)=\theta^{\frac1a}(1+o(1))$, due to \eqref{0ojlfw-COmdf-15}.\end{itemize}
Therefore 
$$g(\theta)=\begin{cases}y^a(\theta)=\theta^a(1+o(1)) & {\mbox{ when }}a\in(0,1) ,\\
\theta(1+o(1)) & {\mbox{ when }}a>1,\end{cases}$$
whence~$$\partial_2 u(1,0^+)
= g'(0^+)=
\begin{cases}
+\infty & {\mbox{ when }}a\in(0,1) ,\\
1 & {\mbox{ when }}a>1.\end{cases}$$

\vfill
\end{document}